\documentclass[11pt]{amsart}

\usepackage[T1]{fontenc}
\usepackage[latin1]{inputenc}

\usepackage[a4paper]{geometry}

\usepackage{amssymb}
\usepackage{amsmath}
\usepackage{amsthm}
\usepackage[all, 2cell, poly]{xypic}
\SelectTips{cm}{10}

\usepackage[hidelinks]{hyperref}

\usepackage{version}
\usepackage{mathrsfs}

\include{macros}

\let\nbd\nobreakdash

\newenvironment{tparagr*}[1]%
  {\noindent\textbf{#1.}}
  {}

\let\tilde\widetilde
\newcommand\opbox{\operatorname{\square}}
\newcommand\opboxp{\operatorname{\square'}}

\newcommand\orth{\operatorname{\pitchfork}}
\newcommand\Sat{\operatorname{Sat}}
\newcommand\bbs[2]{\langle{#1}\bs #2\rangle}
\newcommand\bslash[2]{\langle{#1}\slash #2\rangle}

\newcommand\Map{\operatorname{Map}}

\newcommand\Deltat[1]{\widetilde{\Delta_{#1}}}

\newcommand\horn[2]{\Lambda_{#1}^{#2}}
\newcommand\mhorn[2]{h_{#1}^{#2}}
\newcommand\bDelta[1]{\bord{\Delta_{#1}}}
\newcommand\mbDelta[1]{\delta_{#1}}

\newcommand\bord[1]{\partial #1}
\newcommand\mbord[1]{\delta_{#1}}
\newcommand\mbordz{\delta}

\newcommand\canseg[1]{\partial^{\,#1}}
\newcommand\cansegd[2]{\partial^{\,#1}_{#2}}
\newcommand\bordseg[1]{\partial #1}
\newcommand\mbordseg[1]{\delta_{#1}}

\newcommand\Wp{\W\,'}
\newcommand\Wvert{\W_{\textrm{vert}}}
\newcommand\Whor{\W_{\textrm{hor}}}
\newcommand\WfRezk{\W_{\textrm{fRezk}}}
\newcommand\WRezk[1][n]{\W_{\textrm{Rezk}_n}}
\newcommand\WJoyal{\W_{\textrm{Joyal}}}
\newcommand\WDelta{\W_\Delta}
\newcommand\WQCat[1][n]{\W_{\textrm{QCat}_n}}
\newcommand\WpQCat[1][n]{\W_{\textrm{pQCat}_n}}

\newcommand\Real{\textrm{Real}}
\newcommand\Sing{\textrm{Sing}}

\newcommand\cM{\mathfrak{M}}
\newcommand\An{\textsf{An}}
\renewcommand\W{\textsf{W}}
\newcommand\clC{C}
\newcommand\clD{D}

\newcommand\clL{L}
\newcommand\clI{\mathcal{I}}
\newcommand\clJ{\mathcal{J}}

\author{Dimitri Ara}

\address{Dimitri Ara, Institut de Mathématiques de Jussieu, Université Paris
Diderot~-- Paris 7, Case 7012, Bâtiment Chevaleret, 75205 Paris Cedex 13,
France}
\email{ara@math.jussieu.fr}
\urladdr{http://people.math.jussieu.fr/~ara/}


\dedicatory{À Antoine}

\title{Higher quasi-categories vs higher Rezk spaces}

\begin{document}

\begin{abstract}
We introduce a notion of $n$-quasi-categories as fibrant objects of a model
category structure on presheaves on Joyal's $n$-cell category $\Theta_n$.
Our definition comes from an idea of Cisinski and Joyal. However, we
show that this idea has to be slightly modified to get a reasonable notion.
We construct two Quillen equivalences between the model category of
$n$-quasi-categories and the model category of Rezk $\Theta_n$-spaces,
showing that $n$-quasi-categories are a model for $(\infty, n)$-categories.
For $n = 1$, we recover the two Quillen equivalences defined by Joyal and
Tierney between quasi-categories and complete Segal spaces.
\end{abstract}

\maketitle

\section*{Introduction}

Quasi-categories provide a simplicial model for a certain type of weak
\oo-categories. They were initially introduced as ``simplicial sets
satisfying the restricted Kan condition'' by Boardman and Vogt in
\cite{BoardmanVogt}, and their theory has been developed extensively by
Joyal in \cite{JoyalQCatKan}, \cite{JoyalNQCat} and \cite{JoyalThQCat}, and
Lurie in \cite{LurieHTT} and \cite{LurieHA}, among others.

Weak \oo-categories are a generalization of the strict \oo-categories
that can be defined algebraically as consisting of sets of $k$-arrows in all
dimensions, with compositions and identities at all levels, satisfying some standard
axioms. In a weak \oo-category, these axioms are only asked to be satisfied
``up to coherence'', that is, up to higher arrows themselves satisfying some
identities up to higher arrows, and so on. For instance, weak \oo-categories
of dimension~$2$ can be defined as Bénabou's bicategories \cite{Benabou},
and in this case, the coherences are given by an ``associator'' and two
``unitors'' that have to satisfy Mac Lane's pentagon and triangle axioms.

Quasi-categories only model $(\infty, 1)$-categories, that is, weak
\oo-categories in which the $k$-arrows, for $k>1$, are invertible up to
higher arrows. These $(\infty, 1)$-categories play an important role in
homotopy theory as one can associate an $(\infty, 1$)\nbd-category to any
abstract homotopy theory. For instance, from the homotopy theory of spaces,
we get an $(\infty, 1)$\nbd-cat\-e\-go\-ry where objects are spaces,
$1$-arrows are maps of spaces, $2$-arrows are homotopies, $3$-arrows are
homotopies of homotopies, and so on. Notice that $k$-arrows, for $k
\ge 1$, are indeed invertible in this \oo-category since homotopies are
invertible up to higher homotopies.  More generally, an $(\infty,
n)$-category is a weak \oo-category in which all $k$-arrows are weakly
invertible for $k>n$.

A priori, to make sense of this definition, one first has to make precise
the notion of a weak \oo-category. This was done by Grothendieck
in~\cite{GrothPS} (actually, Grothendieck defined a notion of
$\infty$-groupoid but Maltsiniotis noticed in \cite{MaltsiGr} and
\cite{MaltsiCat} that his definition can be adapted).  Variations on this
definition are studied by the author in \cite{AraThesis}, \cite{AraHomGr}
and~\cite{AraStrWeak} (see also \cite{MaltsiGrCat}). There are now plenty of
competing definitions (see for instance \cite{LeinsterSurvey}
and~\cite{Leinster} for some of them).

A second approach, which is the one this paper is about, is to define
directly $(\infty, n)$\nbd-cat\-e\-go\-ries, without references to weak
\oo-categories, by means of homotopical algebra. There are numerous
such models for $(\infty, 1)$-categories. The four most
popular are probably those explained in Bergner's survey
\cite{BergnerSurvey}, namely, quasi-categories, complete Segal spaces, Segal
categories and categories enriched in Kan complexes.  These four models are
equivalent in the following sense: Joyal (\cite{JoyalThQCat}), Rezk
(\cite{RezkSegSp}), Hirschowitz and Simpson (\cite{SimpsonHirsch}) and
Bergner~(\cite{BergnerSimpCat}) constructed Quillen model category
structures for which the fibrant objects are precisely these four classes of
objects, and these model category structures were shown to be Quillen
equivalent by Bergner (\cite{BergnerThreeModels}),
Joyal~(\cite{JoyalSimpCat}) and Joyal and Tierney
(\cite{JoyalTierneyQCatSeg}). Another model, relative categories, was
introduced and compared to the other models by Barwick and Kan
in~\cite{BarKanRelCat}.

Several of these models for $(\infty, 1)$-categories have been generalized
to models for $(\infty, n)$-categories. Hirschowitz and Simpson introduced a
notion of higher Segal categories in \cite{SimpsonHirsch}. This notion is
the main topic of the book \cite{SimpsonHTHC} of Simpson. In \cite{RezkThSp}
and~\cite{RezkThSpCorr}, Rezk introduced a notion of higher Segal spaces. We
will call these objects Rezk $\Theta_n$-spaces in this paper.  Another model
based on Rezk $\Theta_n$-spaces has been introduced by Bergner and Rezk in
\cite{BergnerRezkComp}. A second generalization of complete Segal spaces
called $n$-fold Segal spaces has been introduced by Barwick (see Section 12
of \cite{BarSchPrUnicity}). The model of relative categories has been
generalized by Barwick and Kan in \cite{BarKanNRelCat}. Several of these
models are explained in Bergner's survey \cite{BergnerSurveyN}.

\smallskip

In this paper, we introduce a notion of $n$-quasi-categories for $n \ge 1$.
Our notion is based on an idea of Cisinski and Joyal. In Section 45 of
\cite{JoyalNQCat}, Joyal writes that he and Cisinski conjecture that the
$\Theta_n$-localizer generated by some kind of higher Segal maps gives rise
to a model for $(\infty, n)$-categories. Let us briefly explain the
terminology. If $A$~is a small category, an (accessible) $A$-localizer is a
class of morphisms of presheaves on~$A$ which is the class of weak
equivalences of a combinatorial model category structure on presheaves on
$A$ whose cofibrations are the monomorphisms. This notion is here applied to
$\Theta_n$, the $n$-truncation of Joyal's cell category introduced in
\cite{JoyalTheta}. In this paper, we show that the idea of Cisinski and
Joyal has to be slightly modified. More precisely, we exhibit equivalent
$n$-categories which are not weakly equivalent in the sense of Cisinski and
Joyal.

We suggest a modification consisting of adding new generators. These
generators are essentially the same as the ones given by Rezk to define his
$\Theta_n$-spaces.  We obtain this way a model category structure on
presheaves on $\Theta_n$ and we define $n$-quasi-categories as the fibrant
objects of this model category. For $n = 1$, by a Theorem of Joyal, we
recover the usual notion of quasi-categories.

We then show that the model category of Rezk $\Theta_n$-spaces is in some
sense canonically associated to our model category of $n$-quasi-categories.
More precisely, we show that the localizer of Rezk $\Theta_n$-spaces is the
simplicial completion in the sense of Cisinski of the localizer of
$n$-quasi-categories. We deduce from this fact, using Cisinski's theory
of simplicial completion, the existence of two Quillen equivalences between
the model category of $n$-quasi-categories and the model category of Rezk
$\Theta_n$-spaces. In particular, the homotopy categories of
$n$-quasi-categories and of Rezk $\Theta_n$-spaces are equivalent. For~$n =
1$, we recover the two Quillen equivalences between quasi-categories and
complete Segal spaces given by Joyal and Tierney in
\cite{JoyalTierneyQCatSeg}.

The tools we use in this work are of two kinds. First, we use the machinery
of $A$\nbd-localizers developed by Cisinski in \cite{Cisinski}. In particular,
our work relies heavily on the notion of simplicial completion of a
localizer.  For $n = 1$, this theory simplifies the work of Joyal and
Tierney. Second, we use the techniques that Joyal and Tierney have developed
to prove that their two adjunctions between quasi-categories and complete
Segal spaces are Quillen equivalences. Many of our arguments are very
similar (if not identical) to the ones used in their proofs. We have tried to
make that clear by citing very precisely their work.

\smallskip

After we made public the first version of this paper, we were informed that
J.~Hahn has obtained related results in his ongoing Ph.D. thesis under the
supervision of Barwick and that Gindi has developed a related homotopy
theory of $\Theta$-sets. Since then, Gindi's work~\cite{GindiWCat} has been
made public.

\smallskip

\begin{tparagr*}{Organization of the paper}
In Section~\ref{sec:1}, we introduce some preliminary terminology. In
Section~\ref{sec:2}, we give a short introduction to Cisinski's theory of
$A$-localizers, which is the language we will use throughout this paper. In
particular, we present the notion of simplicial completion of an
$A$-localizer. As explained above, this notion will play a crucial role in
this paper. Everything from this section is extracted from~\cite{Cisinski}.
In Section~\ref{sec:3}, we introduce tools developed by Joyal and Tierney in
\cite{JoyalTierneyQCatSeg}. Unfortunately, we will need these tools in a
more general setting than the one used in ibid. Nevertheless, everything
adapts trivially and we do not claim any originality for this section. In
Section~\ref{sec:4}, we study the simplicial completion of an $A$-localizer
when $A$ is a regular skeletal Reedy category. The techniques of this
section are still based on \cite{JoyalTierneyQCatSeg} even though they have
to be adapted since we do not have a notion of ``mid anodyne map'' in this
context. For this purpose, we introduce the notion of formal Rezk
$A$-spaces. In Section~\ref{sec:5}, we introduce the
$n$-truncation~$\Theta_n$ of Joyal's cell category as a full subcategory of
the category of strict $n$\nbd-categories and we define our
$\Theta_n$-localizer of $n$-quasi-categories. In Section~\ref{sec:6}, we
explain why the idea of Cisinski and Joyal has to be modified. More
precisely, we show that our new generators, which come from equivalences of
strict $n$\nbd-categories, are not weak equivalences in the sense of
Cisinski and Joyal. In Section~\ref{sec:7}, we explain why nerves of
strict $n$-categories are not $n$-quasi-categories in general (contrary to
what happens for quasi-categories). More precisely, we show that the nerve
of a strict $n$\nbd-category is an $n$-quasi-category if and only if this
$n$-category has no non-trivial invertible $k$-arrows for~$k > 1$. In
Section~\ref{sec:8}, we recall the definition of Rezk $\Theta_n$-spaces and
we show that their localizer is the simplicial completion of our localizer
of $n$-quasi-categories. We obtain two Quillen equivalences between
$n$-quasi-categories and Rezk $\Theta_n$-spaces. We deduce that the model
category of $n$-quasi-categories is cartesian closed from the analogous
result for Rezk $\Theta_n$-spaces. Finally, in \hyperref[sec:A]{an
appendix}, we compare the language of localizers to the language of
Bousfield localization.
\end{tparagr*}

\smallskip

\begin{tparagr*}{Notation}
Let $C$ be a category. The class of objects of
$C$ will be denoted by $\Ob(C)$ and if $X, Y$ is a pair of objects of $C$,
the set of morphisms of $C$ from $X$ to $Y$ will be denoted by~$\Hom_C(X,
Y)$. The opposite category of $C$ will be denoted by~$C^\op$. If $D$ is a
second category, we will denote by $\Homi(C, D)$ the category of functors
from $C$ to $D$.

If $A$ is a small category, the category of presheaves on $A$ will be
denoted by $\smash{\pref{A}}$. If $X$ is a presheaf on $A$ and $a$ is an
object of $A$, we will sometimes write $X_a$ for $X(a)$. We will denote by
$e_{\pref{A}}$ the terminal object of $\smash{\pref{A}}$ and by
$\varnothing_{\pref{A}}$ the initial object of $\smash{\pref{A}}$.

If $u : A \to B$ is a functor between small categories, we will denote by
$u^*$ the functor from $\pref{B}$ to $\pref{A}$ given by precomposition by
$u$. We will denote by $u_!$ its left adjoint and by~$u_*$ its right
adjoint. 

Finally, we will denote by $\Cat$ the category of small categories.
\end{tparagr*}

\section{Preliminaries}\label{sec:1}

In this section, we gather some categorical preliminaries.

\begin{paragr}
Let $\C$ be a category. Let $i : A \to B$ and $p : X \to Y$ be two morphisms
of $\C$. Recall that the morphism $i$ has the \ndef{left lifting property
with respect to $p$} (or that the morphism~$p$ has the \ndef{right lifting
property with respect to $i$}) if for every commutative square
\[
\xymatrix{
A \ar[d]_i \ar[r] & X \ar[d]^p \\
B \ar[r] & Y \pbox{,}
}
\]
there exists a \ndef{lift}, i.e., a morphism $B \to X$ making the two
triangles commute. We will then write $i \orth p$. If the lift is unique,
one says that the lifting property is a \ndef{unique lifting property}.

Let $\clC$ be a class of morphisms of $\C$. We will denote by $l(\clC)$
(resp.~by $r(\clC)$) the class of morphisms having the left lifting property
(resp.~the right lifting property) with respect to~$\clC$ (i.e., with
respect to every morphism of $\clC$). We define the \ndef{saturation}
$\Sat(\clC)$ of $\clC$ as
\[ \Sat(\clC) = lr(\clC). \]
The class $\clC$ is said to be \ndef{saturated} if $\clC = \Sat(\clC)$.

One easily checks that
\[ r(\Sat(\clC)) = r(\clC). \]
Dually, we have $\Sat(l(\clC)) = l(\clC)$. In other words, the class of
morphisms having the left lifting property with respect to a fixed class of
morphisms is saturated. 

If $\C$ is a presheaf category and $S$~is a
\emph{set} of morphisms of $\C$, the small object argument shows that
$\Sat(S)$ is the class of retracts of transfinite compositions of pushouts
of morphisms of $S$.
\end{paragr}

\begin{paragr}\label{paragr:box_cart_clos}
Let $\C$ be a cartesian closed category with an internal $\Hom$ functor $\Homi$
and let $u : A \to B$, $v : C \to D$ and $w : E \to F$ be three
morphisms of $\C$. 
We will denote by
\[
u \times' v : A \times D \amalg_{A \times C} B \times C \to B \times D
\]
the canonical morphism induced by the commutative square
\[
\xymatrix{
A \times C \ar[r] \ar[d] & A \times D \ar[d] \\
B \times C \ar[r] & B \times D \pbox{;} \\
}
\]
and by 
\[
\Homi'(v, w) : \Homi(D, E) \to \Homi(C, E) \times_{\Homi(C, F)} \Homi(D,
F)
\]
the canonical morphism induced by the commutative square
\[
\xymatrix{
\Homi(D, E) \ar[d] \ar[r] & \Homi(D, F) \ar[d] \\
\Homi(C, E) \ar[r] & \Homi(C, F) \pbox{.}
}
\]
By adjunction (see for instance Proposition 7.6 of
\cite{JoyalTierneyQCatSeg}), we have
\[ 
u \times' v \orth w
\quad\Leftrightarrow\quad
u \orth  \Homi'(v, w).
\]
\end{paragr}

\begin{paragr}\label{paragr:def_triv_fib}
Let $A$ be a small category. A \ndef{cellular model} of $\pref{A}$ is a
\emph{set} $\cM$ of monomorphisms of~$\pref{A}$ such that $\Sat(\cM)$ is the
class of monomorphisms of $\pref{A}$. Such a cellular model always exists by
Proposition~1.2.27 of \cite{Cisinski}.

A morphism of $\pref{A}$ is a \ndef{trivial fibration} if it has the right
lifting property with respect to monomorphisms of $\pref{A}$. If $\cM$ is a
cellular model of $\pref{A}$, then a morphism of $\pref{A}$ is a trivial
fibration if and only if it has the right lifting property with respect to
$\cM$.
\end{paragr}

\begin{paragr}
A \ndef{Reedy category} is a category $A$ endowed with two subcategories
$A_+$ and $A_-$ satisfying the following properties:
\begin{enumerate}
\item there exists a map $d : \Ob(A) \to \mathbb{N}$, assigning to every object
of $A$ an non-negative integer, such that:
\begin{enumerate}
\item if $a \to a'$ is a morphism of $A_+$ which is not an identity,
then $d(a) < d(a')$;
\item if $a \to a'$ is a morphism of $A_-$ which is not an identity,
then $d(a) > d(a')$;
\end{enumerate}
\item every morphism of $A$ factors uniquely as a morphism of $A_-$
followed by a morphism of $A_+$.
\end{enumerate}
We will often denote a Reedy category simply by its underlying category.

Let $A$ be a Reedy category. It is obvious that $A$ contains no non-trivial
automorphisms. Moreover, if $f$ is a monomorphism (resp.~an epimorphism) of
$A$, then $f$ belongs to $A_+$ (resp.~to $A_-$).

A Reedy category $A$ is said to be \ndef{skeletal} if it satisfies the
following additional property:
\begin{enumerate}
\item[(3)] every morphism of $A_-$ admits a section; two parallel
morphisms of $A_-$ are equal if and only if they admit the same set of
sections.
\end{enumerate}

A skeletal Reedy category $A$ is said to be \ndef{regular} if it satisfies
the following additional property:
\begin{enumerate}
\item[(4)] every morphism of $A_+$ is a monomorphism.
\end{enumerate}
These two notions come from Chapter 8 of \cite{Cisinski} where they are
called normal skeletal categories and regular skeletal categories.

By a remark above, if $A$ is a regular skeletal Reedy category, then $A_+$
is exactly the class of monomorphisms of $A$. In particular, being a regular
skeletal Reedy category is a property of a category and not an additional
structure.

Let $A$ be a Reedy category and let $a$ be an object of $A$. We will denote
by $\bord{a}$ the subpresheaf of $a$ obtained by taking the union of the
images of all the morphisms $a' \to a$ of~$A_+$ different from the identity. We
will denote by $\mbord{a} : \bord{a} \to a$ the inclusion morphism.
\end{paragr}

\begin{prop}\label{prop:cell_mod_Reedy}
Let $A$ be a skeletal Reedy category. Then the set
\[ \{ \mbord{a} : \bord{a} \to a;\,\, a \in \Ob(A) \} \]
is a cellular model of $\pref{A}$.
\end{prop}

\begin{proof}
Since $A$ is a Reedy category, it contains no non-trivial automorphisms. The
result thus follows from Proposition 8.1.37 of \cite{Cisinski}. 
\end{proof}

\begin{paragr}
Let $\Delta$ be the simplex category. Recall that its objects are the
ordered sets
\[ \Delta_n = \{0, \dots, n\},\quad n \ge 0, \]
and its morphisms are the order preserving maps between them.
The category $\Delta$ carries a Reedy category structure where $\Delta_+$ is
the set of injections and $\Delta_-$ is the set of surjections. This Reedy
category structure is regular skeletal. For $n \ge 0$, we will denote by
$\delta_n$ the morphism $\delta_{\Delta_n} : \bDelta{n} \to \Delta_n$ of
simplicial sets (i.e., of presheaves on $\Delta$).

Let $n \ge 1$ and let $k$ such that $0 \le k \le n$. Recall that the
\ndef{horn} $\horn{n}{k}$ is the sub-simplicial set obtained by taking the
union of the images of all the injections $\Delta_m \to \Delta_n$ except the
identity and the unique injection $\Delta_{n-1} \to \Delta_n$ avoiding $k$.
We will denote by~$\mhorn{n}{k} : \horn{n}{k} \to \Delta_n$ the inclusion
morphism.

Set
\[ \Lambda = \{\mhorn{n}{k}; \,\, n \ge 1, \,\, 0 \le k \le n\}. \]
The class of \ndef{simplicial anodyne extensions} is the saturation of the
set $\Lambda$. A morphism of simplicial sets is a \ndef{Kan fibration} if it
has the right lifting property with respect to $\Lambda$ and hence with
respect to every simplicial anodyne extension.

Recall that the category of simplicial sets admits a combinatorial model
category structure, defined by Quillen in \cite{Quillen}, in which
cofibrations are the monomorphisms and fibrations are the Kan fibrations. We
will call the weak equivalences of this model category the \ndef{simplicial
weak homotopy equivalences}.
\end{paragr}

\section{Cisinski's theory of $A$-localizers}\label{sec:2}

The purpose of this section is to introduce Cisinski's theory of
$A$-localizers. In the language of localizers, the main theorem of this
paper can be stated by saying that the localizer of Rezk
$\Theta_n$-spaces is the simplicial completion of the localizer of
$n$-quasi-categories.

Let us explain roughly what this means. If $A$ is a small category, an
$A$-localizer is a class $\W$ of maps of $\pref{A}$ satisfying some
conditions that are satisfied when $\W$ is the class of weak equivalences of
a model category structure on~$\pref{A}$ whose cofibrations are the
monomorphisms. Conversely, by a theorem of Cisinski, if $\W$ is an
(accessible) $A$-localizer, then $\W$ is the class of weak equivalences of
such a model category structure on $\pref{A}$.

To every $A$-localizer $\W$, there is an associated $(A \times
\Delta)$-localizer $\WDelta$ called the simplicial completion of $\W$.
By a theorem of Cisinski, the model category structures associated to~$\W$
and $\WDelta$ are Quillen equivalent. In particular, what we called above
the main theorem of this paper implies that $n$-quasi-categories and
Rezk $\Theta_n$-spaces are Quillen equivalent.

\medskip

Throughout the section, we fix a small category $A$.

\begin{paragr}
An \ndef{$A$-localizer} is a class $\W$ of morphisms of $\pref{A}$
such that the following conditions hold:
\begin{enumerate}
\item $\W$ satisfies the 2-out-of-3 property;
\item every trivial fibration of $\pref{A}$ is in $\W$;
\item the class of monomorphisms of $\pref{A}$ belonging to $\W$ is stable
under pushout and transfinite composition.
\end{enumerate}
If $\W$ is an $A$-localizer, the elements of $\W$ will be called
\ndef{$\W$-equivalences}. It is immediate that an intersection of
$A$-localizers is again an $A$-localizer. If $\clC$ is a class of morphisms
of~$\pref{A}$, the \ndef{$A$-localizer generated by $\clC$} is by definition
the intersection of all the $A$-localizers containing~$\clC$. We will denote
it by $\W(\clC)$. A localizer is \ndef{accessible} if it is generated by a
\emph{set}.
\end{paragr}

\begin{thm}\label{thm:mod_str_loc}
Let $\W$ be a class of morphisms of $\pref{A}$. Then the following
conditions are equivalent:
\begin{enumerate}
\item $\W$ is an accessible $A$-localizer;
\item there exists a cofibrantly generated model category structure on
$\pref{A}$ whose weak equivalences are the elements of $\W$ and whose
cofibrations are the monomorphisms.
\end{enumerate}
\end{thm}

\begin{proof}
See Theorem 1.4.3 of \cite{Cisinski}.
\end{proof}

\begin{paragr}
We will denote by $\W_\infty$ the $\Delta$-localizer of simplicial weak
homotopy equivalences. Simplicial weak homotopy equivalences will thus also
be called $\W_\infty$-equivalences.
\end{paragr}

\begin{paragr}
By the above theorem, from an accessible $A$-localizer $\W$, we obtain a
model category structure on $\pref{A}$. We will call this model category
structure the \ndef{$\W$-model category structure}. The weak equivalences of
the $\W$-model category structure are the elements of $\W$ and the cofibrations
are the monomorphisms. The fibrations (resp.~the fibrant objects) will be
called \ndef{$\W$-fibrations} (resp.~\ndef{$\W$-fibrant objects}).
\end{paragr}

\begin{rem}
The language of localizers is very related to the language of left
\hbox{Bousfield} localization. In particular, we prove in an appendix to
this paper that if $\W$ is an accessible $A$\nbd-localizer and $\Wp$ is an
accessible $A$\nbd-localizer generated by $\W$ and a class of
morphisms~$\clC$, then the $\Wp$-model category is the left Bousfield
localization of the $\W$\nbd-model category with respect to $\clC$.
\end{rem}

\begin{paragr}
We will say that a localizer $\W$ is \ndef{cartesian} if it is closed under
binary product.
\end{paragr}

\begin{prop}\label{prop:loc_cart_eq}
Let $\W$ be an accessible $A$-localizer. The following conditions are
equivalent:
\begin{enumerate}
\item the localizer $\W$ is cartesian;
\item the class of monomorphisms of $\smash{\pref{A}}$ belonging to $\W$ is closed
under binary product;
\item the $\W$-model category is cartesian closed.
\end{enumerate}
\end{prop}

\begin{proof}
By definition, the class of monomorphisms of $\pref{A}$ belonging to $\W$ is
the class of trivial cofibrations of the $\W$-model category. The
equivalence $(1) \Leftrightarrow (2)$ thus follows immediately from the fact
that monomorphisms and trivial fibrations are stable under product.

Let us prove the implication $(2) \Rightarrow (3)$. Let $U \to V$ and $S \to
T$ be two monomorphisms of $\pref{A}$. Consider the diagram
\[
\xymatrix{
U \times S \ar[r] \ar[d] & U \times T \ar[d] \ar@/^/[ddr] \\
V \times S \ar[r] \ar@/_/[drr] & U \times T \amalg_{U \times S} V \times S \ar[dr] \\
& & V \times T \pbox{.} \\
}
\]
The morphism $U \times T \amalg_{U \times S} V \times S \to V \times T$ is a
monomorphism (it is nothing but the inclusion of $U \times T \cup V \times
S$ into $V \times T$). Suppose moreover that the morphism $U \to V$ is a 
$\W$\nbd-equivalence. Then $U \times S \to V \times S$ and~$U \times T \to
V \times T$ are trivial cofibrations by $(2)$. It follows that $U \times T
\to U \times T \amalg_{U \times S} V \times S$ is a trivial cofibration and,
by the 2-out-of-3 property, that $U \times T \amalg_{U \times S} V \times S
\to V \times T$ is a $\W$-equivalence, thereby proving~$(3)$.

Let us show the converse. Let $U \to V$ be a trivial cofibration and let $T$
be a presheaf on $A$. It clearly suffices to show that $U \times T \to V
\times T$ is a trivial cofibration. By applying~$(3)$ to $U \to V$ and to
the unique morphism $\varnothing_{\pref{A}} \to T$, we obtain that the
morphism $U \times T \amalg_{U \times \varnothing} V \times
\varnothing_{\pref{A}} \to V \times T$ is a trivial cofibration. But this
morphism is nothing but $U \times T \to V \times T$.
\end{proof}

\begin{prop}\label{prop:loc_cart_clos}
Let $\W$ be an $A$-localizer, let $\Wp$ be an $A'$-localizer and let $F :
\pref{A} \to \pref{A'}$ be a functor. Suppose that $F$ respects binary
products and that $\W = F^{-1}(\Wp)$. Then if the localizer $\Wp$ is
cartesian, so is the localizer $\W$.
\end{prop}

\begin{proof}
This is an immediate consequence of the definition of cartesian localizers.
\end{proof}

\begin{paragr}
An \ndef{interval of $\pref{A}$} consists of a presheaf $I$ on $A$
and two morphisms $\canseg{0}, \canseg{1} : \smash{e_{\pref{A}}} \to I$. We will
often denote such an interval simply by $I$. If $I$ is an interval of
$\pref{A}$, we will denote by $\{\epsilon\}$, where $\epsilon = 0, 1$, the image
of $\canseg{\epsilon}$ in $I$. We will denote by $\bordseg{I}$ the union
$\{0\} \cup \{1\}$ and by $\mbordseg{I} : \bordseg{I} \to I$ the canonical
inclusion. If $X$ is a presheaf on $A$ and $\epsilon = 0, 1$, we will
denote by $\cansegd{\epsilon}{X} : X \to X \times I$ the morphism $X \times
\canseg{\epsilon}$.

An interval $I$ is \ndef{separating} if the intersection of $\{0\}$ and
$\{1\}$ is $\smash{\varnothing_{\pref{A}}}$. The interval $I$ is \ndef{injective} if
the morphism $I \to e_{\pref{A}}$ is a trivial fibration of $\pref{A}$.
\end{paragr}

\begin{paragr}\label{paragr:an_ext}
Let $I$ be a separating interval of $\pref{A}$. A class of \ndef{anodyne
$I$-extensions} is a class $\An$ of monomorphisms of $\pref{A}$ satisfying
the following conditions:
\begin{enumerate}
\item there exists a \emph{set} $S$ such that $\An = \Sat(S)$;
\item the canonical inclusion 
  \[ U \times I \cup V \times \{\epsilon\} \to V \times I \]
  is in $\An$ for every monomorphism $U \to V$ of $\pref{A}$
  and~$\epsilon = 0, 1$;
\item the canonical inclusion
  \[ U \times I \cup V \times \bordseg{I} \to V \times I \]
  is in $\An$ for every $U \to V$ in $\An$.
\end{enumerate}

Let $S$ be a set of monomorphisms of $\pref{A}$. By Proposition 1.3.13 of
\cite{Cisinski}, there exists a smallest class of anodyne $I$-extensions
containing $S$. We will denote this class by $\An_I(S)$ and we will call
its elements \ndef{anodyne $(S, I)$-extensions}.

The class of anodyne $(S, I$)-extensions can be described in the following way.
We define by induction on $k \ge 0$ a set $\Lambda^k_I(S)$ of monomorphisms
of $\pref{A}$ by setting
\[
  \Lambda^0_I(S) = S, \quad \Lambda^{k+1}_I(S) = \Lambda_I(\Lambda^k_I(S)),
\]
where
\[ \Lambda_I(T) = \{U \times I \cup V \times \bordseg{I} \to V \times I;\,\,
U \to V \in T\}. \]
By definition, the set $\Lambda^\infty_I(S)$ is the union of the
$\Lambda^k_I(S)$, $k \ge 0$. Let $\cM$ be any cellular model of $\pref{A}$.
Then the class of anodyne $(S, I)$-extensions is the saturated class generated by
\[
  \Lambda^\infty_I(S) \cup 
  \{U \times I \cup V \times \{\epsilon\} \to V \times I;
    \,\, U \to V \in \cM,\,\, \epsilon = 0,1\}.
\]
Note that this description is not exactly the one given in paragraph 1.3.12
of \cite{Cisinski}. Nevertheless, it follows easily from Remark 1.3.15 of
ibid.~that it is correct.
\end{paragr}

\begin{rem}
By Section 2 of Chapter IV of \cite{GZ}, in the case where $A = \Delta$, $I
= \Delta_1$ and $S$ is empty, the class of anodyne $(S, I)$-extensions is
precisely the class of simplicial anodyne extensions. See also paragraph
2.1.3 of \cite{Cisinski}.
\end{rem}

\begin{lemma}\label{lemma:an_ext}
Let $S$ be a set of monomorphisms of $\pref{A}$ and let $I$ be a
separating interval. Let $\clC$ be a class of monomorphisms of $\pref{A}$
satisfying the following conditions:
\begin{enumerate}
\item $\clC$ contains $S$;
\item $\clC$ is saturated;
\item if $u : X \to Y$ and $v : Y \to Z$ are monomorphisms of $\pref{A}$
such that $vu$ and $u$ are in $\clC$, then $v$ is in $\clC$;
\item the morphisms $\cansegd{\epsilon}{X} : X \to X \times I$ belong to
  $\clC$
for every presheaf $X$ on $A$ and $\epsilon = 0, 1$.
\end{enumerate}
Then $\clC$ contains the class of anodyne $(S, I)$-extensions.
\end{lemma}

\begin{proof}
See Lemma 1.3.16 of \cite{Cisinski}.
\end{proof}

\begin{paragr}
Let $S$ be a set of monomorphisms of $\pref{A}$ and let $I$ be a separating
interval of $\pref{A}$. A morphism of $\pref{A}$ will be said to be a
\ndef{naive $(S, I)$-fibration} if it has the right lifting property with
respect to the class of anodyne $(S, I)$-extensions. A presheaf $X$ on
$A$ will be said to be \ndef{$(S, I)$-fibrant} if the morphism $X \to
e_{\pref{A}}$ is a naive $(S, I)$-fibration.
\end{paragr}

\begin{thm}\label{thm:naive_fib}
Let $S$ be a set of monomorphisms of $\pref{A}$ and set $\W = \W(S)$.
Let $J$ be an injective separating interval of $\pref{A}$. Then, if $f$ is a
morphism whose target is a $\W$-fibrant object, then $f$ is a $\W$-fibration
if and only if it is a naive $(S, J)$-fibration. In particular, the class of
$\W$-fibrant objects and of $(S, J)$-fibrant objects coincide.
\end{thm}

\begin{proof}
By Corollary 1.4.18 of \cite{Cisinski}, the $\W$-model category is
the model category associated to $S$ and $J$ in the sense of Theorem 1.3.22
of ibid. The result then follows from Proposition~1.3.36 of ibid.
\end{proof}

From now on, we fix an $A$-localizer $\W$.

\begin{paragr}
We will denote by
\[
A \xleftarrow{p} A \times \Delta \xrightarrow{q} \Delta
\]
the two canonical projections. They induce functors
\[
\pref{A} \xrightarrow{p^*} \pref{A \times \Delta}
\xleftarrow{q^*} \pref{\Delta}.
\]
These functors admit left and right adjoints and hence respect limits and
colimits. In particular, they preserve monomorphisms.

We will denote by
\[
i_0 : A \to A \times \Delta
\]
the functor defined by
\[ i_0(a) = (a, \Delta_0). \]
It follows from the fact that $\Delta_0$ is a terminal object of $\Delta$
that the functor
\[ i^*_0 : \pref{A \times \Delta} \to \pref{A} \]
is right adjoint to $p^*$.
\end{paragr}

\begin{paragr}\label{paragr:def_simp_comp}
We will say that a morphism $f : X \to Y$ of $\pref{A \times \Delta}$ is a
\ndef{horizontal equivalence} if for every $n \ge 0$, the morphism
$f_{\bullet, n} : X_{\bullet, n} \to Y_{\bullet, n}$ is a $\W$-equivalence.
We will denote by~$\Whor$ the class of horizontal equivalences. If follows
from the existence of the injective model category structure for
combinatorial model categories that if $\W$ is accessible, then $\Whor$ is
an accessible $(A \times \Delta)$-localizer.

We will say that a morphism $f : X \to Y$ of $\pref{A \times \Delta}$ is a
\ndef{vertical equivalence} if for every object $a$ of $A$, the morphism
$f_{a, \bullet} : X_{a, \bullet} \to Y_{a, \bullet}$ is a
$\W_\infty$-equivalence. We will denote by~$\Wvert$ the class of vertical
equivalences. If follows again from the existence of the injective model
category structure for combinatorial model categories that if $\W$ is
accessible, then $\Wvert$ is an accessible $(A \times \Delta)$-localizer.

The \ndef{simplicial completion} of $\W$ is the $(A \times \Delta)$-localizer
generated by 
\[ \Whor \cup \{X \times q^*(\Delta_1) \to X;\,\, X \in \Ob(\pref{A
\times \Delta})\}. \]
We will denote it by $\WDelta$.
\end{paragr}

\begin{rem}\label{rem:hor_vert}
To make sense of the terminology ``vertical weak equivalences'' and ``horizontal
weak equivalences'', one has to think of a presheaf on $A \times \Delta$ as
a grid of sets whose columns are indexed by objects of $A$ and whose rows
are indexed by integers $n \ge 0$. 
\end{rem}

\begin{prop}
If the $A$-localizer $\W$ is accessible, then the $(A \times
\Delta)$-localizer $\WDelta$ is accessible. 
\end{prop}

\begin{proof}
See Proposition 2.3.24 of \cite{Cisinski}.
\end{proof}

\begin{prop}\label{prop:reg_compl_simpl}
If $A$ is a regular skeletal Reedy category, then the simplicial
completion of $\W$ is generated by $\Whor$ and $\Wvert$.
\end{prop}

\begin{proof}
This is a direct consequence of Proposition 8.2.9 and Corollary~3.4.37 of
\cite{Cisinski}.
\end{proof}

\begin{prop}\label{prop:Qeq_pA}
Suppose $\W$ is accessible. Then the adjunction
\[
p^* : \pref{A} \rightleftarrows \pref{A \times \Delta} : i^*_0,
\]
where $\pref{A}$ (resp.~$\pref{A \times \Delta}$)
is endowed with the $\W$-model category structure (resp.~with the
$\WDelta$-model category structure), is a Quillen equivalence. 
\end{prop}

\begin{proof}
We have already noticed that the functor $p^*$ preserves monomorphisms and
hence cofibrations. It also preserves weak equivalences by definition of
$\WDelta$. The pair $(p^*, i^*_0)$ is hence a Quillen adjunction. By
Proposition 2.3.27 of \cite{Cisinski}, the functor $p^*$ induces an
equivalence on the homotopy categories. The pair $(p^*, i^*_0)$ is hence a
Quillen equivalence. 
\end{proof}

\begin{coro}\label{coro:loc_cart_simpl}
Suppose $\W$ is accessible. If the simplicial completion of $\W$ is
cartesian, then so is the localizer $\W$.
\end{coro}

\begin{proof}
The functor $p^*$ respects limits and in particular binary products. Since
every object is cofibrant in the $\W$-model category structure, it follows from the
above proposition that it preserves and reflects weak equivalences. The
result thus follows from Proposition~\ref{prop:loc_cart_clos}.
\end{proof}

\begin{paragr}
Let $D$ be a cosimplicial object in $\pref{A}$, i.e., a functor $D : \Delta
\to \pref{A}$. For $n
\ge 0$, we will denote $D(\Delta_n)$ by $D_n$.

Consider the functor 
\[ \underline{D} : A \times \Delta \to \pref{A} \]
defined by
\[ \underline{D}(a, \Delta_n) = a \times D_n. \]
Since $\pref{A}$ is cocomplete, this functor induces an adjunction
\[
\Real_D : \pref{A \times \Delta} \rightleftarrows \pref{A} : \Sing_D,
\]
where $\Real_D$ is the unique extension of $\underline{D}$ to $\pref{A \times
\Delta}$ which respects colimits, and $\Sing_D$~is defined by
\[ \Sing_D(X)_{a, n} = \Hom_{\pref{A}}(a \times D_n, X), \]
where $a$ is an object of $A$ and $n \ge 0$.

The cosimplicial object $D$ is a \ndef{cosimplicial $\W$-resolution} if it
satisfies the following conditions:
\begin{enumerate}
\item the morphism $D_0 \amalg D_0 \to D_1$ induced by $\mbDelta{1} :
  \bDelta{1} = \Delta_0 \amalg \Delta_0 \to \Delta_1$ is a monomorphism;
\item for every $n \ge 0$ and every presheaf $X$ on $A$, the
canonical projection $X \times D_n \to X$ is a $\W$-equivalence.
\end{enumerate}
\end{paragr}

\begin{prop}\label{prop:Qeq_RealD}
Suppose $\W$ is accessible and let $D$ be a cosimplicial $\W$-resolution.
Then the adjunction
\[
\Real_D : \pref{A \times \Delta} \rightleftarrows \pref{A} : \Sing_D,
\]
where $\pref{A \times \Delta}$ (resp.~$\pref{A}$) is endowed with the
$\WDelta$-model category structure (resp.~with the $\W$-model category
structure), is a Quillen equivalence. 
\end{prop}

\begin{proof}
By Lemma 2.3.10 of \cite{Cisinski}, the functor $\Real_D$ preserves
monomorphisms and hence cofibrations. By Proposition 2.3.27 of ibid., it
also preserves weak equivalences. The pair~$(\Real_D, \Sing_D)$
is hence a Quillen pair. By the same proposition, the functor~$\Real_D$
induces an equivalence on the homotopy categories. The pair $(\Real_D,
\Sing_D)$ is hence a Quillen equivalence.
\end{proof}

\section{The vertical and the horizontal model categories}\label{sec:3}

Let $A$ be a skeletal Reedy category and let $\W$ be an accessible 
$A$-localizer. We have two different Reedy model category structures on the
category 
\[ \pref{A \times \Delta} = \Homi(A^\op, \pref{\Delta}) = \Homi(\Delta^\op,
\pref{A}). \]
One is coming from the Reedy structure of $A$ and the classical model
category structure on simplicial sets; the second is coming from the Reedy
structure of $\Delta$ and the model category structure associated to the
localizer $\W$ (see Theorem~\ref{thm:mod_str_loc}). Following Section 2
of~\cite{JoyalTierneyQCatSeg}, the former model category structure we will
called the vertical model category structure, and the latter will be called
the horizontal model category structure (see Remark~\ref{rem:hor_vert} on
this terminology). The purpose of this section is to introduce and study
these two model category structures. (Our case of interest in this paper is
the case where $A =  \Theta_n$ and $\W$ is the localizer of
$n$-quasi-categories defined in paragraph~\ref{paragr:loc_nqcat}.)

\medskip

Throughout the section, we fix two small categories $A$ and $B$. (We will soon
specialize to the case $B = \Delta$.)

\begin{paragr}
We will denote by
\[
\begin{split}
\opbox : \pref{A} \times \pref{B} & \to \pref{A \times B} \\
\end{split}
\]
the functor defined in the following way: if $X$ is a presheaf on $A$, $Y$ is a
presheaf on $B$, $a$ is an object of $A$ and $b$ is an object of $B$,
then
\[
(X \opbox Y)_{a, b} = X_a \times Y_b.
\]
If $X$ is a fixed presheaf on $A$, then the functor
\[
X \opbox - : \pref{B} \to \pref{A \times B}
\]
admits a right adjoint
\[
X\bs - : \pref{A \times B} \to \pref{B}
\]
given by
\[ (X\bs Z)_b = \Hom_{\pref{A\times B}}(X \opbox b, Z), \]
where $Z$ is a presheaf on $A \times B$ and $b$ is an object of $B$.
Similarly, if $Y$ is a fixed presheaf on $B$, the functor
\[
- \opbox Y : \pref{A} \to \pref{A \times B}
\]
admits a right adjoint
\[
- \slash Y : \pref{A \times B} \to \pref{A}.
\]
given by
\[ (Z\slash Y)_a = \Hom_{\pref{A\times B}}(a \opbox Y, Z), \]
where $Z$ is a presheaf on $A \times B$ and $a$ is an object of $A$.

Thus, if $X$ is a presheaf on $A$, $Y$ is a presheaf on $B$ and $Z$ is a
presheaf on $A \times B$, we have natural bijections
\[
\Hom_{\pref{A\times B}}(X \opbox Y, Z) \cong \Hom_{\pref{B}}(Y, X\bs Z)
\cong \Hom_{\pref{A}}(X, Z\slash Y).
\]
\end{paragr}

\begin{paragr}
Let $u : U \to V$ be a morphism of $\pref{A}$ and let $v : S \to T$ be
a morphism of $\pref{B}$. We will denote by
\[
u \opboxp v : U \opbox T \amalg_{U \opbox S} V \opbox S \to V \opbox T
\]
the morphism induced by the commutative square
\[
\xymatrix{
U \opbox S \ar[r] \ar[d] & U \opbox T \ar[d] \\
V \opbox S \ar[r] & V \opbox T \pbox{.} \\
}
\]
If $f : X \to Y$ is a morphism of $\pref{A \times B}$, we will denote
by
\[ \bbs{u}{f} : V\bs X \to V\bs Y \times_{U\bs Y} U\bs X \]
the morphism induced by the commutative square
\[
\xymatrix{
V \bs X \ar[r] \ar[d] & V \bs Y \ar[d] \\
U \bs X \ar[r] & U \bs Y \pbox{;}
}
\]
and by
\[ \bslash{f}{v} : X\slash T \to Y\slash T \times_{Y \slash S} X\slash S \]
the morphism induced by the commutative square
\[
\xymatrix{
X \slash T \ar[r] \ar[d] & Y \slash T \ar[d] \\
X \slash S \ar[r] & Y \slash S \pbox{.}
}
\]
\end{paragr}

\begin{prop}\label{prop:equiv_orth}
If $u$ is a morphism of $\pref{A}$, $v$ is a morphism of $\pref{B}$ and $f$
is a morphism of $\pref{A \times B}$, then we have
\[
(u \opboxp v) \orth f
\quad\Leftrightarrow\quad
u \orth \bslash{f}{v}
\quad\Leftrightarrow\quad
v \orth \bbs{u}{f}.
\]
\end{prop}

\begin{proof}
See Proposition 7.6 of \cite{JoyalTierneyQCatSeg}.
\end{proof}

From now on, we set $B = \Delta$.

\begin{prop}\label{prop:cell_mod_A_Delta}
If $\cM$ is a cellular model of $\pref{A}$, then
\[
\{ \mbordz \opboxp \mbDelta{n} : U \opbox \Delta_n \amalg_{U\opbox
\bDelta{n}} V \opbox \bDelta{n} \to V \opbox
\Delta_n;\,\, \mbordz : U \to V \in \cM,\,\, n \ge 0 \}
\]
is a cellular model of $\pref{A \times \Delta}$.
\end{prop}

\begin{proof}
See Lemma 2.3.2 of \cite{Cisinski}.
\end{proof}

\begin{prop}\label{prop:triv_fib}
A morphism $f$ of $\pref{A \times \Delta}$ is a trivial fibration if and
only if the following equivalent conditions are satisfied:
\begin{enumerate}
\item $\bbs{\mbordz}{f}$ is a trivial fibration for every $\mbordz$ in a
fixed cellular model of $\pref{A}$;
\item $\bbs{u}{f}$ is a trivial fibration for every monomorphism $u$ of
$\pref{A}$;
\item $\bslash{f}{\mbDelta{n}}$ is a trivial fibration for all $n \ge
0$;
\item $\bslash{f}{v}$ is a trivial fibration for every monomorphism $v$
of simplicial sets.
\end{enumerate}
\end{prop}

\begin{proof}
The proof is essentially the same as the one of Proposition 2.3 of
\cite{JoyalTierneyQCatSeg}. Fix a cellular model $\cM$ of $\pref{A}$.
By the previous proposition, a morphism $f$ of $\pref{A \times \Delta}$ is a
trivial fibration if and only if for every $\mbordz$ in $\cM$ and every $n
\ge 0$, we have $\mbordz \opboxp \mbDelta{n} \orth f$. But by
Proposition~\ref{prop:equiv_orth}, we have
\[
\mbordz \opboxp \mbDelta{n} \orth f
\quad\Leftrightarrow\quad
\mbDelta{n} \orth \bbs{\mbordz}{f}
\quad\Leftrightarrow\quad
\mbordz \orth \bslash{f}{\mbDelta{n}},
\]
thereby proving that $f$ is a trivial fibration if and only if one of the
conditions (1) and~(3) is satisfied. The other equivalences follow from the
fact that the class of morphisms having the left lifting property with
respect to a fixed class of morphisms is saturated.
\end{proof}

\emph{From now on, we assume that $A$ is a skeletal Reedy category.} In particular,
if $a$ is an object of $A$, we have a morphism $\mbord{a} : \bord{a} \to a$.

\begin{paragr}
Recall that we have defined in paragraph \ref{paragr:def_simp_comp} a class
$\Wvert$ of morphisms of $\pref{A \times \Delta}$ called vertical
equivalences. In the language of this section, a morphism $f$ of $\pref{A
\times \Delta}$ is a vertical equivalence if and only if, for every object
$a$ of $A$, the map $a\bs f$ is a $\W_\infty$\nbd-equivalence.
\end{paragr}

\begin{thm}[Reedy, Kan]\label{thm:vert_mcs}
The class $\Wvert$ is an accessible $(A \times \Delta)$-localizer. The
fibrations of the $\Wvert$-model category structure are the morphisms $f$
such that for every object $a$ of $A$, the morphism $\bbs{\mbord{a}}{f}$ is
a Kan fibration. Moreover, this model category structure is proper,
simplicial and cartesian closed.
\end{thm}

\begin{proof}
The proof is essentially the same as the one of Theorem 2.6 of
\cite{JoyalTierneyQCatSeg}. Consider the Reedy model category structure on $\pref{A
\times \Delta}$ seen as $\Homi(A^\op, \pref{\Delta})$, where $\pref{\Delta}$
is endowed with the $\W_\infty$-model category structure. By definition, the
weak equivalences of this model category structure are the vertical
equivalences. Recall that a morphism $f
: X \to Y$ of~$\pref{A \times \Delta}$ is a fibration (resp.~a trivial
fibration) if and only if, for every object $a$ of $A$, the morphism
\[
X_a \to Y_a \times_{M_a(Y)} M_a(X),
\]
where $M_a(Z)$ denotes the $a$-th latching object of a presheaf $Z$, is a
Kan fibration (resp.~a trivial fibration). But this morphism is nothing but
$\bbs{\mbord{a}}{f}$. In particular, by Proposition \ref{prop:triv_fib}, the
trivial fibration of this Reedy structure are the trivial fibrations in the
sense of paragraph~\ref{paragr:def_triv_fib}. It follows that the
cofibrations of this Reedy structure are the monomorphisms.  This Reedy
structure is hence a model category structure whose weak equivalences are
the $\Wvert$-equivalences and whose cofibrations are the monomorphisms. By
Theorem \ref{thm:mod_str_loc}, the class $\Wvert$ is hence an accessible
localizer and the $\Wvert$-model category structure is this Reedy model
category structure.

The fact that this model category structure, which is nothing but the
injective model category structure on simplicial presheaves, is proper,
simplicial and cartesian closed is well-known (see the proof of Theorem 2.6
of \cite{JoyalTierneyQCatSeg} for a proof).
\end{proof}

\begin{paragr}
We will call the $\Wvert$-model category structure on $\pref{A \times
\Delta}$ the \ndef{vertical model category structure}. A fibration of this
structure will be called a \ndef{vertical fibration}.
\end{paragr}

\begin{prop}\label{prop:equiv_vert_fib}
A morphism $f$ of $\pref{A \times \Delta}$ is a vertical fibration
if and only if the following equivalent conditions are satisfied:
\begin{enumerate}
\item $\bbs{\mbord{a}}{f}$ is a Kan fibration for every object $a$ of
$A$;
\item $\bbs{u}{f}$ is a Kan fibration for every monomorphism $u$ of
$\pref{A}$;
\item $\bslash{f}{\mhorn{n}{k}}$ is a trivial fibration for all $n \ge
1$ and $0 \le k \le n$;
\item $\bslash{f}{v}$ is a trivial fibration for every simplicial
anodyne extension $v$.
\end{enumerate}
\end{prop}

\begin{proof}
The proof is essentially the same as the one of Proposition 2.5 of
\cite{JoyalTierneyQCatSeg}. The morphism $\bbs{\mbord{a}}{f}$ is a Kan
fibration if and only if for every $n \ge 1$ and $0 \le k \le n$, we have
$\mhorn{n}{k} \orth \bbs{\mbord{a}}{f}$. But
\[
\mhorn{n}{k} \orth \bbs{\mbord{a}}{f}
\quad\Leftrightarrow\quad
\mbord{a} \orth \bslash{f}{\mhorn{n}{k}},
\]
and the result follows from the fact that the $\mbord{a}$'s form a
cellular model of $\pref{A}$ (Proposition~\ref{prop:cell_mod_Reedy}).
\end{proof}

From now on, we fix an accessible $A$-localizer $\W$.

\begin{paragr}
Recall that we have defined in paragraph \ref{paragr:def_simp_comp} a class
$\Whor$ of morphisms of $\pref{A \times \Delta}$ called horizontal
equivalences.  In the language of this section, a morphism $f$ of $\pref{A
\times \Delta}$ is a horizontal equivalence if and only if, for every $n \ge
0$, the morphism $f/\Delta_n$ is a $\W$\nbd-equivalence.
\end{paragr}

\begin{thm}[Reedy]\label{thm:hor_mcs}
The class $\Whor$ is an accessible $(A \times \Delta)$-localizer. The
fibrations of the $\Whor$-model category structure are the morphisms $f$
such that for every $n \ge 0$, the morphism $\bslash{f}{\mbDelta{n}}$ is a
$\W$-fibration.
\end{thm}

\begin{proof}
The proof is essentially the same as the one of Proposition 2.10 of
\cite{JoyalTierneyQCatSeg}. Consider the Reedy model category structure
on $\pref{A \times \Delta}$ seen as $\Homi(\Delta^\op, \pref{A})$, where
$\pref{A}$ is endowed with the $\W$-model category structure. By definition,
the weak equivalences of this model category structure are the horizontal
equivalences.  For the same reasons as in the proof of
Theorem~\ref{thm:vert_mcs}, a morphism $f$ of $\pref{A \times \Delta}$ is a
fibration (resp.~a trivial fibration) for this Reedy structure if and only
if for every $n \ge 0$, the morphism $\bslash{f}{\mbDelta{n}}$ is a
$\W$-fibration (resp.~a trivial fibration). It follows as in the
proof of Theorem \ref{thm:vert_mcs} that the cofibrations of this Reedy
structure are the monomorphisms.  The class $\Whor$ is hence an accessible
localizer and the $\Whor$-model category structure is this Reedy model
category structure.
\end{proof}

\begin{paragr}
We will call the $\Whor$-model category structure on $\pref{A \times
\Delta}$ the \ndef{horizontal model category structure}. A fibration of this
structure will be called a \ndef{horizontal fibration}.
\end{paragr}

\begin{rem}
All the results of this section have appropriate generalizations when
$\Delta$ is replaced by a skeletal Reedy category $B$ endowed with an
accessible $B$-localizer. This is easily seen once
Proposition~\ref{prop:cell_mod_A_Delta} has been generalized. An inspection
of the proof of this proposition, that is of the proof of Lemma 2.3.2 of
\cite{Cisinski}, reveals that all the properties of~$\Delta$ used in
this proof are shared by skeletal Reedy categories (see Section~8.1
of \cite{Cisinski}).
\end{rem}

\section{The model category of formal Rezk spaces}\label{sec:4}

In this section, we associate to any accessible $A$-localizer $\W$, where
$A$ is a skeletal Reedy category, an $(A\times\Delta)$-localizer $\WfRezk$
of formal Rezk spaces. The purpose of the section is then to show that when
the skeletal Reedy category $A$ is regular, the localizer~$\WfRezk$ is the
simplicial completion of the localizer $\W$.

\medskip

Throughout the section, we fix a skeletal Reedy category $A$, a set $S$ of
monomorphisms of~$\pref{A}$ and an injective separating interval $J$ of
$\pref{A}$. We denote by $W$ the $A$-localizer generated by~$S$.

\begin{paragr}\label{paragr:loc_fRezk}
We will denote by $\WfRezk$ the $(A \times \Delta)$-localizer generated by
\[
\Wvert \cup p^*(S) \cup \{p^*(X \times J \to X);\,\, X \in
\Ob(\pref{A})\}.
\]
Note that by the 2-out-of-3 property, $\WfRezk$ is also generated by
\[
\Wvert \cup p^*(S) \cup 
\{p^*(\cansegd{\epsilon}{X} : X \to X \times J); \,\, X \in \Ob(\pref{A}),\,\,
\epsilon = 0,1\}.
\]
\end{paragr}

\begin{prop}
The $(A \times \Delta)$-localizer $\WfRezk$ is accessible.
\end{prop}

\begin{proof}
By Theorem \ref{thm:vert_mcs}, the localizer $\Wvert$ is accessible. Let $T$
be a generating set of~$\Wvert$. We claim that the set
\[
T \cup p^*(S) \cup
\{ p^*\big(\bord{a} \times J \cup a \times \{\epsilon\} \to a
\times J\big); \,\, a \in \Ob{A},\,\, \epsilon = 0, 1\}
\]
generates the localizer $\WfRezk$.

Let $\W'$ be the localizer generated by this set. We first show that $\W'$
is included in~$\WfRezk$. Let $a$ be an object of $A$ and let $\epsilon =
0,1$. Consider the diagram
\[
\xymatrix{
p^*(\bord{a} \times \{\epsilon\}) \ar[r] \ar[d] &
  p^*(\bord{a} \times J) \ar[d] \ar@/^/[ddr] \\
p^*(a \times \{\epsilon\}) \ar[r]) \ar@/_/[drr] &
  p^*(\bord{a} \times J \cup a \times \{\epsilon\}) \ar[dr] \\
& & p^*(a \times J) \pbox{.}
}
\]
The morphisms
\[
p^*\big(\bord{a} \times \{\epsilon\} \to \bord{a} \times J\big)
\quad\text{and}\quad
p^*\big(a \times \{\epsilon\} \to a \times J\big)
\]
are $\WfRezk$-equivalences. The upper horizontal morphism is hence both a
$\WfRezk$-equiva\-lence and a monomorphism. Since the square of the diagram is
cocartesian, the lower horizontal morphism is also a $\WfRezk$-equivalence.
It follows from the 2-out-of-3 property that the morphism
\[ p^*\big(\bord{a} \times J \cup a \times \{\epsilon\} \to a
\times J\big) \]
is a $\WfRezk$-equivalence.

Let us now show that $\WfRezk$ is included in $\W'$. Let $\clC$ be the class of
monomorphisms~$U \to V$ of $\pref{A}$ such that
\[
p^*\big(U \times J \cup V \times \{\epsilon\} \to V
\times J\big)
\]
belongs to $\W'$ for $\epsilon = 0,1$. It is easy to check that this class
is saturated. But by definition of $\W'$, the class $\clC$ contains a cellular model of
$\pref{A}$. It hence contains every monomorphism of $\pref{A}$. If~$X$ is a
presheaf on $A$, the morphism $f : \varnothing_{\pref{A}} \to X$ thus
belongs to $\clC$. This means that
\[
p^*\big(X \times \{\epsilon\} \to X \times J\big)
\]
is in $\W'$, thereby proving the result.
\end{proof}

\begin{paragr}
We will call the $\WfRezk$-model category structure on $\pref{A \times
\Delta}$ the \ndef{model category structure of formal Rezk $A$-spaces}. A
$\WfRezk$-fibrant object will be called a \ndef{formal Rezk $A$-space}.
\end{paragr}

\begin{prop}
A presheaf $X$ on $A \times \Delta$ is a formal Rezk $A$-space if and
only if it satisfies the following conditions:
\begin{enumerate}
\item $X$ is vertically fibrant;
\item $s \bs X$ is a trivial fibration for every $s$ in $S$;
\item $\cansegd{\epsilon}{Y} \bs X$ is a trivial fibration for every presheaf
$Y$ on $A$ and $\epsilon = 0,1$.
\end{enumerate}
\end{prop}

\begin{proof}
Let
\[
\clC = p^*(S) \cup 
\{p^*(\cansegd{\epsilon}{X} : X \to X \times J); \,\, X \in \Ob(\pref{A}),\,\,
\epsilon = 0,1\}.
\]
By definition, the localizer $\WfRezk$ is generated by $\Wvert$ and $C$. It
follows from Proposition \ref{prop:loc_fib_obj} that the formal Rezk
$A$-spaces are the fibrant $\clC$-local objects of the vertical model
category, i.e., the vertically fibrant objects~$X$ of~$\pref{A
\times \Delta}$ such that for every morphism $f : K \to L$ in $\clC$, the
morphism
\[
\Map(f, X) : \Map(L, X) \to \Map(K, X),
\]
where $\Map$ denotes the simplicial enrichment of $\pref{A \times \Delta}$,
is a $\W_\infty$-equivalence. (Recall that the vertical model category is
a simplicial model category.) But if a morphism~$f$ of $\pref{A \times
\Delta}$ is equal to
$p^*(f_0)$ for some morphism $f_0 : K_0 \to L_0$ of $\pref{A}$, then
$\Map(f, X)$ is nothing but the morphism
\[ f_0 \bs X : L_0 \bs X \to K_0 \bs X. \]
Thus, a presheaf $X$ on $A \times \Delta$ is a formal Rezk $A$-space if and
only if it satisfies the following conditions:
\begin{enumerate}
\item $X$ is vertically fibrant;
\item $s \bs X$ is a $\W_\infty$-equivalence for every $s$ in $S$;
\item $\cansegd{\epsilon}{Y} \bs X$ is a $\W_\infty$-equivalence for every
  presheaf $Y$ on $A$ and $\epsilon = 0,1$.
\end{enumerate}
But by Proposition \ref{prop:equiv_vert_fib}, under the assumption that $X$
is vertically fibrant, the morphisms~$s \bs X$ and~$\cansegd{\epsilon}{Y} \bs
X$, where $s$ is in $S$ and $Y$ is a presheaf on $A$, are Kan fibrations.
They are hence $\W_\infty$-equivalences if and only if they are trivial
fibrations, thereby proving the result.
\end{proof}

\begin{prop}\label{prop:equiv_fRezk}
Let $X$ be a vertically fibrant presheaf on $A \times \Delta$. Then the
following conditions are equivalent:
\begin{enumerate}
\item $X$ is a formal Rezk $A$-space;
\item $u\bs X$ is a trivial fibration for every anodyne
$(S, J)$-extension $u$;
\item $X\slash\delta_n$ is a naive $(S, J)$-fibration for all $n \ge
0$;
\item $X/v$ is a naive $(S, J)$-fibration for every monomorphism
$v$ of simplicial sets.
\end{enumerate}
\end{prop}

\begin{proof}
The proof is similar to the one of Proposition 3.4 of
\cite{JoyalTierneyQCatSeg}.

$2 \Leftrightarrow 3 \Leftrightarrow 4)$ Let $u$ be a morphism of
$\pref{A}$. The morphism $u\bs X$ is a trivial fibration if and only if, for all
$n \ge 0$, we have $\delta_n \orth u\bs X$. But we have
\[
\delta_n \orth u \bs X
\quad\Leftrightarrow\quad
u \orth X/\delta_n.
\]

$1 \Rightarrow 2)$ Let $\clC$ be the class of monomorphisms $u$ of
$\pref{A}$ such that $u\bs X$ is a trivial fibration. We have just seen that
$u$ belongs to $\clC$ if and only if, for all $n \ge 0$, we have $u \orth
X/\delta_n$. The class $\clC$ is thus saturated. Moreover, by Proposition
\ref{prop:equiv_vert_fib}, if $u$ is a monomorphism of~$\pref{A}$, then $u \bs
X$ is a Kan fibration. In particular, $u$ belongs to $\clC$ if and only if $u
\bs X$ is a $\W_\infty$-equivalence. It follows that the class $\clC$
satisfies the 2-out-of-3 property. Since $X$ is a formal Rezk $A$-space,
the morphisms of $S$ and the $\cansegd{\epsilon}{Y}$'s belong to $\clC$ by the
previous proposition. It follows from Lemma \ref{lemma:an_ext} that the
class $\clC$ contains the class of anodyne $(S, J)$-extensions.

$2 \Rightarrow 1)$ By definition, the morphisms of $S$ and the
$\cansegd{\epsilon}{Y}$'s are anodyne $(S, J)$-extensions.
\end{proof}

\begin{coro}
If $X$ is a formal Rezk $A$-space, then for every simplicial set $U$, the
presheaf $X/U$ is $\W$-fibrant. In particular, $X_{\bullet, n}$ is
$\W$-fibrant for every $n \ge 0$.
\end{coro}

\begin{proof}
Consider the morphism of simplicial sets $v : \varnothing_{\pref{\Delta}}
\to U$. By the previous proposition, the morphism $X/v = X/U \to
X/\varnothing_{\pref{\Delta}} = e_{\pref{A}}$ is a naive $(S, J)$-fibration.
By Theorem~\ref{thm:naive_fib}, the object $X/U$ is hence $\W$-fibrant. The
second assertion follows from the fact that $X/\Delta_n = X_{\bullet, n}$.
\end{proof}

\begin{prop}
Let $f : X \to Y$ be a vertical fibration between formal Rezk $A$-spaces.
Then, for every monomorphism $v : S \to T$ of simplicial sets, the morphism
\[ \bslash{f}{v} : X/T \to Y/T \times_{Y/S} X/S \]
is a $W$-fibration between $W$-fibrant objects.
\end{prop}

\begin{proof}
The proof is similar to the one of Proposition 3.10 of
\cite{JoyalTierneyQCatSeg}. Let $v : S \to T$ be a monomorphism of
simplicial sets. By the above corollary, the presheaves $X/S$ and~$X/T$ are
$\W$-fibrant. By proposition \ref{prop:equiv_fRezk}, the morphism $Y/T \to
Y/S$ is a naive $(S, J)$-fibration. It follows that $Y/T \times_{Y/S} X/S
\to X/S$ is a naive $(S, J)$-fibration. The presheaf~$Y/T \times_{Y/S} X/S$
is thus $\W$-fibrant by Theorem \ref{thm:naive_fib}.

By the same theorem, it suffices to show that $\bslash{f}{v}$ is a naive
$(S, J)$-fibration, i.e., that for every anodyne $(S, J)$-extension
$u$, we have $u \orth \bslash{f}{v}$. But we have
\[
u \orth \bslash{f}{v}
\quad\Leftrightarrow\quad
v \orth \bbs{u}{f},
\]
and it thus suffices to show that for every anodyne $(S, J)$-extension $u :
U \to V$, the morphism
\[
\bbs{u}{f} : V\bs X \to V\bs Y \times_{U\bs Y} U\bs X
\]
is a trivial fibration. Let $u$ be such an anodyne $(S, J)$-extension.
Since $f$ is a vertical fibration, by Proposition~\ref{prop:equiv_vert_fib},
the morphism~$\bbs{u}{f}$ is a Kan fibration. It thus suffices to prove that
$\bbs{u}{f}$ is a $\W_\infty$-equivalence. Consider the commutative square
\[
\xymatrix{
V \bs X \ar[r] \ar[d] & V \bs Y \ar[d] \\
U \bs X \ar[r] & U \bs Y \pbox{.}
}
\]
Since $X$ and $Y$ are formal Rezk $A$-spaces, by Proposition
\ref{prop:equiv_fRezk}, the vertical morphisms are trivial fibrations. This
square is hence homotopically cartesian and the result follows.
\end{proof}

\begin{coro}
Let $f : X \to Y$ be a vertical fibration between formal Rezk $A$-spaces.
Then $f$ is a horizontal fibration.
\end{coro}

\begin{proof}
By the previous proposition, for every $n \ge 0$, the morphism
$\bslash{f}{\delta_n}$ is a $\W$\nbd-fibration. The results hence follows from
Theorem \ref{thm:hor_mcs}.
\end{proof}

\begin{thm}
The $(A \times \Delta)$-localizer $\WfRezk$ contains $\Whor$.
\end{thm}

\begin{proof}
The proof is similar to the one of Theorem 4.5 of
\cite{JoyalTierneyQCatSeg}. We have an adjunction
\[ 
F : \pref{A \times \Delta}_{\,\textrm{hor}} \rightleftarrows \pref{A \times
\Delta}_{\,\textrm{fRezk}} : G, 
\]
where $\pref{A \times \Delta}_{\,\textrm{hor}}$ and $\pref{A \times
\Delta}_{\,\textrm{fRezk}}$ denote the $\Whor$-model category
and the $\WfRezk$\nbd-model category, respectively,
and $F, G$ both denote the identity functor.  The functor $F$ clearly
preserves monomorphisms and hence cofibrations.  Moreover, by the above
corollary, the functor $G$ preserves fibrations between fibrant objects. It
follows from a lemma of Dugger (Corollary A.2 of \cite{DuggerReplSimpl})
that $(F, G)$ is a Quillen pair.  In particular, by Ken Brown's lemma,
$F$~preserves weak equivalences between cofibrant objects, and hence all the
weak equivalences since every object of $\pref{A \times
\Delta}_{\,\textrm{hor}}$ is cofibrant. This exactly means that $\WfRezk$
contains $\Whor$.
\end{proof}

\emph{From now on, we suppose that $A$ is a regular skeletal Reedy
category.}

\begin{thm}\label{thm:fRezk_simp_comp}
The $(A \times \Delta)$-localizer $\WfRezk$ is the simplicial completion
of $\W$.
\end{thm}

\begin{proof}
Since $A$ is a regular skeletal Reedy category, by Proposition
\ref{prop:reg_compl_simpl}, the localizer $\WDelta$ is generated by $\Wvert$
and $\Whor$. By definition, $\Wvert$ is included in $\WfRezk$ and by the
above theorem, $\Whor$ is included in $\WfRezk$. The localizer $\WDelta$ is
thus included in $\WfRezk$.

Conversely, let us show that $\WfRezk$ is included in $\WDelta$.
By Proposition \ref{prop:reg_compl_simpl}, $\Wvert$~is included in
$\WDelta$. Moreover, the morphisms of $S$ are in $\W$ by definition and the
morphisms~$X \times J \to X$, where $X$ is a presheaf on $A$, are in $\W$
since $J$ is injective. But by definition, $p^*$ sends $\W$
into $\Whor$ and hence into $\WDelta$. The generators of~$\WfRezk$ are hence
included in $\WDelta$, thereby proving the result.
\end{proof}

\begin{thm}\label{Qeq_fRezk}
Let us endow $\pref{A}$ (resp.~$\pref{A \times \Delta}$) with the
$\W$-model category structure (resp.~with the $\WfRezk$-model category
structure).
\begin{enumerate}
\item 
Then the adjunction
\[
p^* : \pref{A} \rightleftarrows \pref{A \times \Delta} : i^*_0
\]
is a Quillen equivalence.
\item
Let $D : \Delta \to \pref{A}$ be a cosimplicial $\W$-resolution. Then the
adjunction
\[
\Real_D : \pref{A \times \Delta} \rightleftarrows \pref{A} : \Sing_D,
\]
is a Quillen equivalence. 
\end{enumerate}
\end{thm}

\begin{proof}
This follows from Propositions \ref{prop:Qeq_pA} and \ref{prop:Qeq_RealD},
and from the above theorem.
\end{proof}

\section{$n$-quasi-categories}\label{sec:5}

The purpose of this section is to introduce our notion of
$n$-quasi-categories.

\medskip

Throughout the section, we fix an integer $n \ge 1$.

\begin{paragr}
We will denote by $\G_n$ the category generated by the graph
\[
\xymatrix{
\Dn{0} \ar@<.6ex>[r]^-{\Ths{1}} \ar@<-.6ex>[r]_-{\Tht{1}} &
\Dn{1} \ar@<.6ex>[r]^-{\Ths{2}} \ar@<-.6ex>[r]_-{\Tht{2}} &
\cdots \ar@<.6ex>[r]^-{\Ths{n-1}} \ar@<-.6ex>[r]_-{\Tht{n-1}} &
\Dn{n-1} \ar@<.6ex>[r]^-{\Ths{n}} \ar@<-.6ex>[r]_-{\Tht{n}} &
\Dn{n} 
}
\]
under the relations
\[\Ths{i+1}\Ths{i} = \Tht{i+1}\Ths{i}\quad\text{and}\quad\Ths{i+1}\Tht{i} =
\Tht{i+1}\Tht{i}, \qquad 1 \le i < n.\]
For $i, j$ such that $0 \le j \le i \le n$, we will denote by $\Ths[j]{i}$
and $\Tht[j]{i}$ the morphisms from $\Dn{j}$ to~$\Dn{i}$ defined by
\[
  \Ths[j]{i} = \Ths{i}\cdots\Ths{j+2}\Ths{j+1}\quad\text{and}\quad
  \Tht[j]{i} = \Tht{i}\cdots\Tht{j+2}\Tht{j+1}.
\]

By definition, the category of \ndef{$n$-graphs} is the category
$\pref{\G_n}$ of presheaves on $\G_n$. An $n$-graph $X$ thus consists of a
diagram of sets
\[
\xymatrix{
X_{n} \ar@<.6ex>[r]^-{\Gls{n}} \ar@<-.6ex>[r]_-{\Glt{n}} &
X_{n-1} \ar@<.6ex>[r]^-{\Gls{n-1}} \ar@<-.6ex>[r]_-{\Glt{n-1}} &
\cdots \ar@<.6ex>[r]^-{\Gls{2}} \ar@<-.6ex>[r]_-{\Glt{2}} &
X_1 \ar@<.6ex>[r]^-{\Gls{1}} \ar@<-.6ex>[r]_-{\Glt{1}} &
X_0
}
\]
satisfying the relations
\[\Gls{i}\Gls{i+1} = \Gls{i}\Glt{i+1}\quad\text{and}\quad\Glt{i}\Gls{i+1} =
\Glt{i}\Glt{i+1}, \qquad 1 \le i < n.\]

If $X$ is an $n$-graph, we will call $X_0$ the set of \ndef{objects} of $X$
and~$X_k$, for $0 \le k \le n$, the set of \ndef{$k$-arrows} of $X$. If $f$
is a $k$-arrow of $X$ for $k \ge 1$, the $(k-1)$\nbd-arrow~$\Gls{i}(f)$
(resp.~$\Glt{i}(f)$) will be called the \ndef{source} (resp.~the
\ndef{target}) of $f$. We will often denote an arrow $f$ of $X$ whose source
is $x$ and whose target is $y$ by $f : x \to y$.

We will say that two $k$-morphisms $f, g$ of an $n$-graph $X$ are
\ndef{parallel} if, either $k = 0$, or $k \ge 1$ and these morphisms satisfy
\[ \Gls{k}(f) = \Gls{k}(g) \quad\text{and}\quad \Glt{k}(f) = \Glt{k}(g). \]
If $f, g$ is a pair of parallel $k$-arrows for $k < n$, we will denote by
$\Hom_X(f, g)$ the set of $(k+1)$-arrows of $X$ from $f$ to $g$.
\end{paragr}

\begin{paragr}\label{paragr:def_table}
Let $m$ be a positive integer. A \ndef{table of dimensions} of \ndef{width
$m$} consists of a table
\[
\tabdim
\]
filled with integers satisfying
\[ i_k > i'_k\quad\text{and}\quad i_{k+1} > i'_k, \qquad 1 \le k < m. \]
The \ndef{dimension} of a table of dimensions is the greatest integer
appearing in it.

Let $C$ be a category under $\G_n$, i.e., a category endowed with a functor
$F : \G_n \to C$. We will often denote in the same way the objects and
morphisms of $\G_n$ and their image by the functor $F$. Let
\[ T = \tabdim \]
be a table of dimensions of dimension at most $n$. The \ndef{globular sum}
in $C$ associated to $T$ (if it exists) is the iterated pushout
\[ (\Dn{i_1}, \Ths[i'_1]{i_1}) \amalgd{i'_1} (\Tht[i'_1]{i_2}, \Dn{i_2},
\Ths[i'_2]{i_2}) \amalgd{i'_2} \dots
\amalgd{i'_{m-1}} (\Tht[i'_{m-1}]{i_m}, \Dn{i_m}) \]
in $C$, i.e., the colimit of the diagram
\[
\xymatrix@R=.2pc@C=1pc{
\Dn{i_1} &  & \Dn{i_2} &  & \Dn{i_3} &        & \Dn{i_{m - 1}} & & \Dn{i_{m}} \\
  &  &   &  &   & \cdots &     & & \\
  & \Dn{i'_1}
  \ar[uul]^{\Ths[i'_1]{i_1}}
  \ar[uur]_{\negthickspace \Tht[i'_1]{i_2}}  &   &
  \Dn{i'_2}' 
  \ar[uul]^{\Ths[i'_2]{i_2}\negthickspace}
  \ar[uur]_{\Tht[i'_2]{i_3}}  &  &  & &
\Dn{i'_{m-1}} \ar[uul]^{\Ths[i'_{m-1}]{i_{m-1}}}  \ar[uur]_{\Tht[i'_{m-1}]{i_m}}
& &
}
\]
in $C$. We will denote it simply by
\[
\Dn{i_1} \amalgd{i'_1} \Dn{i_2} \amalgd{i'_2} \dots
\amalgd{i'_{m-1}} \Dn{i_m}.
\]

We will always see $\pref{\G_n}$ as a category under $\G_n$ by using the
Yoneda functor. If $T$ is a table of dimensions of dimension at most $n$, we
will denote by $G_T$ the globular sum associated to $T$ in
$\pref{\G_n}$.
\end{paragr}

\begin{ex}
If $T$ is the table of dimensions
\[
\setcounter{MaxMatrixCols}{20}
\left(
\begin{matrix}
2 && 2 && 2 && 3 && 2 && 1 \\
& 1 && 0 && 1 && 1 && 0
\end{matrix}
\right),
\]
then the associated $3$-graph $G_T$ is
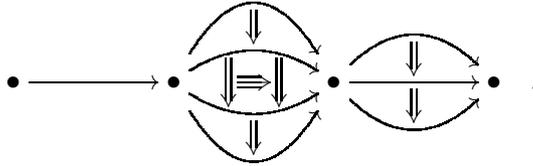
\[
\xymatrix@C=4pc@H=5pc{ 
\bullet
\ar[r]
&
\bullet
\ar@/^2.5pc/[r]|{\vphantom{X}}="g"
\ar@/^1pc/[r]|{\vphantom{X}}="h"
\ar@/_1pc/[r]|{\vphantom{X}}="i"
\ar@/_2.5pc/[r]|{\vphantom{X}}="j"
\ar@{=>}"g";"h"
\ar@<-2ex>@{=>}"h";"i"_{}="0"
\ar@<2ex>@{=>}"h";"i"_{}="1"
\ar@3"0";"1"
\ar@{=>}"i";"j"
&
\bullet
\ar@/^1.5pc/[r]|{\vphantom{X}}="k"
\ar[r]|{\vphantom{X}}="l"
\ar@/_1.5pc/[r]|{\vphantom{X}}="m"
\ar@{=>}"k";"l"
\ar@{=>}"l";"m"
&
\bullet \pbox{.}
&
}
\]
\end{ex}

\begin{rem}
The $G_T$'s are exactly the $n$-graphs associated to finite planar rooted
trees by Batanin in \cite{BataninWCat}.
\end{rem}

\begin{paragr}\label{paragr:n_cat}
We will denote by $\ncat$ the category of strict $n$-categories and strict
$n$-functors. This category can be defined by induction by saying that
$\ncat$ is the category of categories enriched in $\ncat[(n\,-1)]$. Thus, a
strict $n$-category $C$ is given by
\begin{itemize}
  \item a set of objects $\Ob(C)$;
  \item for every pair $x, y$ of objects of $C$, an $(n-1)$-category $\Map_C(x,
    y)$;
  \item for every triple $x, y, z$ of objects of $C$, a strict $(n-1)$-functor
    \[ \comp_0 : \Map_C(y, z) \times \Map_C(x, y) \to \Map_C(x, z); \]
  \item for every object $x$, a distinguished object $\id{x}$ of $\Map_C(x, x)$,
\end{itemize}
satisfying the associativity and unit axioms.  This definition can be
unpacked to a more explicit definition (see for instance Section 1.2 of
\cite{AraMetBrown}).
If $C$ and $D$ are two strict $n$\nbd-categories, a strict $n$\nbd-functor $u : C
\to D$ is given by
\begin{itemize}
  \item a map $u_0 : \Ob(C) \to \Ob(D)$;
  \item for every pair $x, y$ of objects of $C$, a strict $(n - 1)$-functor
  \[
    u_{x, y} : \Map_C(x, y) \to \Map_D(u_0(x), u_0(y)),
  \]
\end{itemize}
satisfying some obvious axioms.
In the sequel, by ``strict $n$-functor'' we will mean ``strict $n$-functor
between strict $n$-categories''.

We will denote by
\[ U_n : \ncat \to \pref{\G_n} \]
the forgetful functor sending a strict $n$-category to its underlying
$n$-graph. We will often implicitly apply this forgetful functor
to transfer notation and terminology from $n$-graphs to strict
$n$-categories. The functor $U_n$ admits a left adjoint
\[ L_n : \pref{\G_n} \to \ncat \]
sending an $n$-graph $G$ to the free strict $n$-category on $G$.

The category $\ncat$ will always be seen as a category under $\G_n$ by using
the functor~$L_n$. In particular, for $k \le n$ we have a strict $n$-category
$\Dn{k}$. Note that $\Dn{0}$ is the terminal object of $\ncat$.

Recall that the category $\ncat$ is cartesian closed. If $C$ and $D$ are two
$n$-categories, we will denote by $\Homi(C, D)$ the corresponding internal
$\Hom$. A $k$-arrow of $\Homi(C, D)$ is given by a strict $n$-functor $C
\times \Dn{k} \to D$. In particular, the set of objects of $\Homi(C, D)$ is
in canonical bijection with $\Hom_{\ncat}(C, D)$.
\end{paragr}

\begin{paragr}
We will denote by $\Theta_n$ the category defined in the following way:
\begin{itemize}
  \item the objects of $\Theta_n$ are the tables of dimensions of dimension
    at most $n$;
  \item if $S$ and $T$ are two objects of $\Theta_n$, then
    \[ \Hom_{\Theta_n}(S, T) = \Hom_{\ncat}(L_n(G_S), L_n(G_T));  \]
  \item the composition and the identities of $\Theta_n$ are induced by
    those of $\ncat$.
\end{itemize}
By definition of $\Theta_n$, we have a fully faithful functor $\Theta_n \to
\ncat$ sending an object $T$ of $\Theta_n$ to~$L_n(G_T)$. This functor is injective on
objects and thus identifies $\Theta_n$ to a full subcategory of $\ncat$.

The functor
\[ \G_n \to \pref{\G_n} \xrightarrow{L_n} \ncat \]
factors through $\Theta_n$ and the category $\Theta_n$ will always be seen as a
category under $\G_n$ by using this functor.
It follows from the fact that $L_n$ commutes with colimits that if
\[ T = \tabdim \]
is an object of $\Theta_n$, then 
\[
  T = 
 \Dn{i_1} \amalgd{i'_1} \Dn{i_2} \amalgd{i'_2} \dots
\amalgd{i'_{m-1}} \Dn{i_m},
\]
where the globular sum is taken in $\Theta_n$.

For $n = 1$, the category $\Theta_1$ is canonically isomorphic to the simplex category
$\Delta$. The object $\Delta_m$ corresponds to the table of dimensions of width
$m$
\[
\left(
\begin{matrix}
1 && 1 &&  \cdots && 1 \cr
& 0 && 0 & \cdots & 0
\end{matrix}
\right)
\]
and we indeed have
\[ \Delta_m = \Delta_1 \amalg_{\Delta^{}_0} \dots \amalg_{\Delta^{}_0} \Delta_1 \]
in $\Delta$. In the sequel, we will identify $\Theta_1$ and $\Delta$.
\end{paragr}

\begin{rem}
The category $\Theta_n$ is canonically isomorphic to (a truncation of) the cell
category introduced by Joyal in \cite{JoyalTheta}. This was proved
independently by Makkai and Zawadowski in \cite{MakZaw} and by Berger in
\cite{BergerNerve}. Alternative definitions of this category are
given in \cite{BergerNerve} and \cite{BergerWreath}. See also
Proposition~3.11 of \cite{AraThtld} for a definition by universal property.
\end{rem}

\begin{paragr}
Since the category $\ncat$ is cocomplete, the inclusion functor $\iota :
\Theta_n \to \ncat$ induces an adjunction
\[
\tau_n :  \pref{\Theta_n} \rightleftarrows \ncat : N_n,
\]
where $\tau_n$ is the unique extension of $\iota$ to $\pref{\Theta_n}$ preserving
colimits and $N_n$ is given by the formula 
\[ N_n(C)_T = \Hom_{\ncat}(\iota(T), C), \]
where $C$ is a strict $n$-category and $T$ is an object of~$\Theta_n$.
It follows formally from the fact that $\iota : \Theta_n \to \ncat$ is fully
faithful and that $\ncat$ is cartesian closed  that the functor~$\tau_n$
commutes with binary products (see the proof of Proposition~B.0.15 of
\cite{JoyalThQCat}). Moreover, abstract (but non-trivial)
considerations (see Example 4.24 of~\cite{WeberFam}) show that the
functor~$N_n$ is fully faithful.

For $n = 1$, the functor $N_1$ is the usual nerve functor $N : \Cat \to
\pref{\Delta}$.
\end{paragr}

\begin{rem}
The fact that $N_n$ is fully faithful was first proved by Berger
starting from a combinatorial definition of $\Theta_n$
(see Theorem~1.12 of~\cite{BergerNerve}).
\end{rem}

\begin{thm}[Berger]\label{thm:ThReedy}
The category $\Theta_n$ is a regular skeletal Reedy category.
\end{thm}

\begin{proof}
By Lemma 2.4 and Remark 2.5 of \cite{BergerNerve}, there is a structure of
skeletal Reedy category on $\Theta_n$. It is not hard to show that this
structure is regular (i.e., that the level preserving cellular operators in
Berger's terminology are monomorphisms).
\end{proof}

\begin{rem}
Another point of view on the Reedy structure on $\Theta_n$ can be found
in~\cite{BergnerRezkReedy}.
\end{rem}

\begin{paragr}
Since $\Theta_n$ is a Reedy category, for every object $T$ we have
a presheaf $\bord{T}$ on $\Theta_n$ endowed with a monomorphism
$\mbord{T} : \bord{T} \to T$. The presheaf $\bord{T}$ is obtained by taking
the union of the images of all the monomorphisms $S \to T$ of $\Theta_n$
except the identity.

Note that a morphism of $\Theta_n$ is a monomorphism if and only if its
underlying $n$-graph morphism is a monomorphism, that is, if and only if it
induces injections on $k$-arrows for every $k$ such that $0 \le k
\le n$.
\end{paragr}
 
\begin{paragr}
The category $\pref{\Theta_n}$ will always be seen as the category under
$\G_n$ by using the functor
\[ \G_n \to \Theta_n \to \pref{\Theta_n}. \]
Let
\[ T = \tabdim \]
be an object of $\Theta_n$. We will denote by $I_T$ the globular sum
\[
  \Dn{i_1} \amalgd{i'_1} \Dn{i_2} \amalgd{i'_2} \dots
  \amalgd{i'_{m-1}} \Dn{i_m}
\]
taken in $\pref{\Theta_n}$. There is a canonical morphism
\[
  i_T : I_T \to T
\]
in $\pref{\Theta_n}$ coming from the universal property of $I_T$. One easily
checks that this morphism is a monomorphism.

When $n = 1$, the object $I_{\Delta_k}$ will be denoted by $I_k$. This is the
sub-simplicial set of~$\Delta_k$ obtained by taking the union of all the
$1$-simplices of $\Delta_k$ whose vertices are consecutive integers. This
object is called the \ndef{spine} of $\Delta_k$ by Joyal in
\cite{JoyalThQCat}. We will denote by~\hbox{$i_k : I_k \to \Delta_k$} the inclusion
morphism.
\end{paragr}

\begin{rem}
Let $T$ be an object of $\Theta_n$. The restriction of the morphism $I_T
\to T$ of~$\pref{\Theta_n}$ to~$\pref{\G_n}$ is nothing but the inclusion
morphism $G_T \to U_nL_n(G_T)$ given by the unit of the
adjunction $(L_n, U_n)$.
\end{rem}

\begin{paragr}
Let $k \ge 1$ and let $C$ be a strict $(k-1)$-category. We define a strict
$k$-category~$\Delta_1 \wr C$ as a category enriched in strict $(k
- 1)$-categories in the following way:
\begin{itemize}
  \item the objects of $\Delta_1 \wr C$ are $0$ and $1$;
  \item for every objects $\epsilon$ and $\epsilon'$ of $\Delta_1 \wr C$, we have
    \[ \Map_{\Delta_1 \wr C}(\epsilon, \epsilon') = 
      \begin{cases}
        C & \text{if $\epsilon = 0$ and $\epsilon' = 1$},\\
        \ast & \text{if $\epsilon = \epsilon'$},\\
        \varnothing & \text{if $\epsilon = 1$ and $\epsilon' = 0$}.\\
      \end{cases}
    \]
\end{itemize}
A priori, we have only defined a graph enriched in strict $(k-1)$-categories.  It
is obvious that there is a unique structure of enriched category on this
enriched graph and the strict $n$-category $\Delta_1 \wr C$ is thus
well-defined. The construction $\Delta_1 \wr C$ is clearly functorial in
$C$.

Let $J$ be the simply connected groupoid on two objects $0$ and
$1$. In other words, $J$ is defined in the following way:
\begin{itemize}
  \item the objects of $J$ are $0$ and $1$;
  \item for every objects $\epsilon$ and $\epsilon'$ of $J$, we have
    $\Hom_J(\epsilon, \epsilon') = \ast$.
\end{itemize}
We will denote by $\bordseg{J}$ the discrete subcategory of $J$ consisting
of the objects $0$ and $1$, and by $\mbordseg{J} : \partial{J} \to J$ the
inclusion functor.

We define by induction on $k \ge 1$ a strict $k$-category $J_k$ in the following
way:
\[ 
  J_1 = J \quad\text{and}\quad J_{k} = \Delta_1 \wr J_{k - 1},
  \quad k \ge 2.
\]
This $k$-category is equipped with a strict $k$-functor $j_k : J_k \to \Dn{k-1}$.
For $k = 1$, the functor~$j_1$ is the unique functor $J \to \Dn{0}$. 
We will also denote this functor simply by $j$. For $k \ge 2$, we have
\[ \Dn{k-1} = \Delta_1 \wr \Dn{k-2}. \]
This allows us to define $j_k$ by induction setting
\[
  j_k = \Delta_1 \wr j_{k-1} : \Delta_1 \wr J_{k-1} \to \Delta_1 \wr
  \Dn{k-2}.
\]
This $k$-functor admits two sections $s^0_k$ and $s^1_k$. For $k = 1$ and
$\epsilon = 0, 1$, the section $s^\epsilon_1$ corresponds to the object
$\epsilon$ of $J$. It will also be denoted by $\canseg{\epsilon}$. For $k
\ge 2$ and $\epsilon = 0, 1$, we define $s^\epsilon_k$ by induction in the
following way:
\[
  s^\epsilon_k = \Delta_1 \wr s^\epsilon_{k-1} : \Delta_1 \wr \Dn{k-2} \to
  \Delta_1 \wr J_{k-1}.
\]
\end{paragr}

\begin{rem}
Here are (the underlying graph without the identities of) $J_1$, $J_2$ and $J_3$:
\[ 
\UseAllTwocells
J_1 = \xymatrix{0 \ar@/^2ex/[r] & \ar@/^2ex/[l] 1}
\quad,\quad
    J_2 = \xymatrix@C=3pc@R=3pc{0 \ar@/^2.5ex/[r]_{}="0"
    \ar@/_2.5ex/[r]_{}="1"
    \ar@<-1ex>@2"0";"1" \ar@<-1ex>@2"1";"0"
      &  1}
\quad\text{and}\quad
    J_3 = \xymatrix@C=3pc@R=3pc{0 \ar@/^3ex/[r]_(.47){}="0"^(.53){}="10"
    \ar@/_3ex/[r]_(.47){}="1"^(.53){}="11"
    \ar@<2ex>@2"0";"1"_{}="2" \ar@<-2ex>@2"10";"11"^{}="3"
    \ar@<1ex>@3{->}"2";"3"_{} \ar@<1ex>@3"3";"2"_{}
    &  1 \pbox{.}}
  \]
The two arrows of maximal dimension are inverse of each other. The two
sections $s^\epsilon_k$ correspond to the two non-trivial $(k-1)$-arrows of
$J_k$. The $k$-functor $j_k$ is the unique strict $k$-functor from $J_k$ to $D_{k-1}$
which sends these two $(k-1)$-arrows to the unique non-trivial $(k-1)$-arrow
of $\Dn{k-1}$.
\end{rem}

\begin{paragr}\label{paragr:loc_nqcat}
Let
\[ \clI_n = \{i_T :\,\, I_T \to T;\,\, T \in \Ob(\Theta_n)\}  \]
and
\[
   \clJ_n = \{N_n(j_k) : N_n(J_k) \to N_n(\Dn{k-1});\,\,1 <  k \le n\}.
\]
The \ndef{localizer of $n$-quasi-categories} is the
$\Theta_n$-localizer generated by $\clI_n$ and $\clJ_n$. We will denote it
by $\WQCat$. By the 2-out-of-3 property, this localizer is also generated 
by~$\clI_n$ and $\clJ'_n$ where
\[
   \clJ'_n = \{N_n(s^\epsilon_k) : N_n(\Dn{k-1}) \to N_n(J_k);
   \,\,1 < k \le n, \,\,\epsilon = 0, 1\}.
\]
The $\WQCat$-model category structure on $\pref{\Theta_n}$ will be called the
\ndef{model category of $n$-quasi-categories}. By definition, an
\ndef{$n$-quasi-category} is a $\WQCat$-fibrant object.
\end{paragr}

\begin{rem}
We will show in the next section that $N_n(j) : N_n(J) \to N_n(\Dn{0})$ is a
trivial fibration and hence belongs to any $\Theta_n$-localizer, and that,
on the contrary, none of the morphisms of $\clJ_n$ belong to the
$\Theta_n$-localizer generated by $\clI_n$.
\end{rem}

\begin{paragr}
Recall that a simplicial set $X$ is a \ndef{quasi-category} if the unique
morphism from $X$ to $\Delta_0$ has the right lifting property with respect to
$\mhorn{n}{k} : \horn{n}{k} \to \Delta_n$ for every $n \ge 2$ and every~$0 <
k < n$. Joyal defined in \cite{JoyalThQCat} (see Theorem~6.12) a model
category structure on simplicial sets, the so-called \ndef{model category of
quasi-categories}. This model category is uniquely defined by the fact that
its cofibrations are the monomorphisms and its fibrant objects are the
quasi-categories.
\end{paragr}

\begin{thm}[Joyal]
The model category of $1$-quasi-categories coincide with the model category
of quasi-categories.
\end{thm}

\begin{proof}
Let us denote by $\WJoyal$ the $\Delta$-localizer associated to the model
category of quasi-categories. By Proposition 2.13 of \cite{JoyalThQCat}, the
morphisms of $\clI_1$ are $\WJoyal$-equivalences. The localizer $\WQCat[1]$
is thus contained in $\WJoyal$. It follows that the $\WJoyal$-fibrant
objects, i.e., the quasi-categories, are $\WQCat[1]$-fibrant objects.

Let us show the converse. It suffices to prove that for every $n \ge 2$ and
every $0 < k < n$, the morphism $\mhorn{n}{k} : \horn{n}{k} \to \Delta_n$
is a $\WQCat[1]$-equivalence. This follows from Lemma~3.5
of~\cite{JoyalTierneyQCatSeg} applied to the class of trivial cofibrations of
the $\WQCat[1]$-model category.

We have shown that the two model categories have the same fibrant objects.
Since they also have the same cofibrations, they coincide.
\end{proof}

\begin{rem}
If $C$ is a category, it is well-known that its nerve $N_1(C)$ is a
quasi-category. As we will see in Section~\ref{sec:7}, it is not true that
the nerve of a strict $n$-category is an $n$-quasi-category in general. For
instance, $N_n(J_k)$ is not an $n$-quasi-category for $1 < k \le n$ (see
Corollary~\ref{coro:J_k_not_fib}). This is the reason why we have chosen the
terminology ``$n$\nbd-quasi-category'' rather than ``quasi-$n$-category'' which
was used in a preliminary version of this paper: strict $n$\nbd-categories
should be quasi\nbd-$n$\nbd-categories.
\end{rem}

\section{On our generators of the localizer of
$n$-quasi-categories}\label{sec:6}

In this section, we study our generators of the localizer of
$n$-quasi-categories. We first show that the morphism $N_n(j) : N_n(J) \to
N_n(\Dn{0})$ is a trivial fibration and hence belongs to any
$\Theta_n$-localizer. This is the reason why in dimension $1$ (the case
$\Theta_1 = \Delta$) the spines are sufficient to generate the localizer of
quasi-categories. This is probably also the reason why Cisinski and Joyal
conjectured that the higher spines would generate a $\Theta_n$-localizer
which would model $(\infty, n)$-categories. We show in this section that it
is not the case: more precisely, we show that none of the morphisms of
$\clJ_n$ (which are equivalences of strict $n$-categories and hence should
be equivalences of $(\infty, n)$-categories) belong to the
$\Theta_n$-localizer generated by $\clI_n$.

\medbreak

Throughout the section, we fix an integer $n \ge 1$.

\begin{paragr}\label{paragr:trunc}
Consider the inclusion functor $i : \Cat \to \ncat$. This functor admits a
left adjoint $t : \ncat \to \Cat$ and a right adjoint $t_r : \ncat \to
\Cat$. The functor $t$ will be called the \ndef{truncation functor}.
It sends a strict $n$-category $C$ to the category whose objects are the
same as those of $C$ and whose arrows are the $1$-arrows of $C$ up to
$2$-arrows. The functor~$t_r$ will be called the \ndef{right truncation functor}. It
sends a strict $n$-category to the category whose objects and $1$-arrows are
the same as those of $C$.

The adjunction
\[
  t : \ncat \rightleftarrows \Cat : i
\]
restricts to an adjunction
\[
  t : \Theta_n \rightleftarrows \Delta : i.
\]
This new adjunction induces a third adjunction
\[
  t^* : \pref{\Delta} \rightleftarrows \pref{\Theta_n} : i^*.
\]
An immediate calculation shows that the square
\[
\xymatrix{
  \pref{\Delta} \ar[r]^{t^*} & \pref{\Theta_n} \\
  \Cat \ar[u]^{N_1} \ar[r]_i & \ncat \ar[u]_{N_n}
}
\]
commutes.

On the contrary, the functor $t_r : \ncat \to \Cat$ does not restrict to a
functor $\Theta_n \to \Delta$. Nevertheless, we can consider the functor
\[ \Theta_n \to \ncat \xrightarrow{t_r} \cat \to \pref{\Delta}. \]
Since $\pref{\Delta}$ is cocomplete, this functor induces an adjunction
\[ {t_r}_! : \pref{\Theta_n} \rightleftarrows \pref{\Delta} : t^*_r,\]
where ${t_r}_!$ is the unique extension of this functor $\Theta_n \to
\pref{\Delta}$ to $\pref{\Theta_n}$ preserving
colimits and $t^*_r$ is given by the formula 
\[ t^*_r(X)_T = \Hom_{\pref{\Delta}}(t_r(T), X), \]
where $X$ is a simplicial set and $T$ is an object of~$\Theta_n$.
The functor ${t_r}_!$ and $i^*$ both preserve colimits and coincide on
objects of $\Theta_n$. It follows that they are isomorphic and hence that their
right adjoints are isomorphic. In particular, if $X$ is a simplicial set and
$T$ is an object of $\Theta_n$, we have
\[ i_*(X)_T \cong t^*_r(X)_T = \Hom_{\pref{\Delta}}(t_r(T), X). \]
\end{paragr}

\begin{prop}
Let $f$ be a morphism of simplicial sets. Then $i_*(f)$ is a trivial
fibration of $\pref{\Theta_n}$ if and only if $f$ is a trivial fibration of
simplicial sets.
\end{prop}

\begin{proof}
The functor $i^*$ admits a left adjoint and hence preserves monomorphisms.
It follows that its right adjoint $i_*$ preserves trivial fibrations.

The same argument shows that $i^*$ preserves trivial fibrations (its left
adjoint $t^*$ admits a left adjoint). Suppose $i_*(f)$ is a trivial
fibration. We have just seen that $i^*i_*(f)$ is a trivial fibration.  But
since $i$ is fully faithful, we have $i^*i_*(f) \cong f$ and so $f$ is a
trivial fibration.
\end{proof}

\begin{paragr}
We will say that a category is a \ndef{preorder} if there is at most one
arrow between every pair of objects.
\end{paragr}

\begin{coro}\label{coro:i_n_m}
Let $u$ be a functor between preorders. Then $N_n(u)$ is a trivial fibration
of~$\pref{\Theta_n}$ if and only if $N_1(u)$ is a trivial fibration of
simplicial sets.
\end{coro}

\begin{proof}
If $u$ is any functor, by paragraph \ref{paragr:trunc}, we have
$t^*N_1(u) = N_n(u)$. On the other hand, if $C$ is a preorder, we have
\[ t^*N_1(C)_T = \Hom_\Cat(t(T), C) \cong \Hom_\Cat(t_r(T), C) \cong i_*N_1(C)_T \]
for every object $T$ of $\Theta_n$. If follows that if $u$ is a functor
between preorders, we have
\[ N_n(u) \cong i_*N_1(u) \]
and the result follows from the above proposition.
\end{proof}

\begin{prop}
Let $u$ be a functor. Then $N_1(u)$ is a trivial fibration of
simplicial sets if and only if $u$ is an equivalence of categories
surjective on objects.
\end{prop}

\begin{proof}
The morphism $N_1(u)$ is a trivial fibration if and only if
for every $n \ge 0$, we have~$\mbDelta{n} \orth N_1(u)$. By adjunction, we have
\[
\mbDelta{n} \orth N_1(u)
\quad\Leftrightarrow\quad
\tau_1(\mbDelta{n}) \orth u.
\]
But it is well-known that $\tau_1(\mbDelta{n})$ is an isomorphism for $n \ge 3$
(see for instance the lemma page 32 of \cite{GZ}). Thus, the morphism $N_1(u)$ is a
trivial fibration if and only if it satisfies
\[
\tau_1(\mbDelta{n}) \orth u, \quad n = 0, 1, 2.
\]
The condition for $n = 0$ (resp.~$n = 1$, resp.~$n = 2$) is equivalent to
the surjectivity on objects of $u$ (resp.~the fullness of $u$, resp.~the
faithfulness of $u$), thereby proving the result.
\end{proof}

\begin{coro}\label{coro:fib_triv_fun}
Let $u$ be a functor between preorders. Then $N_n(u)$ is a trivial fibration
of~$\pref{\Theta_n}$ if and only if $u$ is an equivalence of categories
surjective on objects.
\end{coro}

\begin{proof}
This follows from Proposition~\ref{coro:i_n_m} and the above proposition.
\end{proof}

\begin{coro}\label{coro:J_inj}
The morphism $N_n(j) : N_n(J) \to N_n(\Dn{0})$ is a trivial fibration of
$\pref{\Theta_n}$.
\end{coro}

\begin{proof}
The functor $J \to \Dn{0}$ is an equivalence of categories surjective on
objects between preorders. The result thus follows from the above corollary.
\end{proof}

\begin{paragr}\label{paragr:J}
Consider $N_n(J)$ endowed with the two morphisms
\[
  \canseg{\epsilon} = N_n(\canseg{\epsilon}) : N_n(\Dn{0}) \to N_n(J),\quad 
   \epsilon = 0, 1.
\]
The presheaf $N_n(\Dn{0})$ is the terminal object of $\pref{\Theta_n}$ and
$N_n(J)$ is thus endowed with the structure of an interval. It is immediate
that this interval is separating. Moreover, by the above lemma, this
interval is injective.
\end{paragr}

\begin{prop}\label{prop:iso_tau_i_T}
For every object $T$ of $\Theta_n$, the $n$-functor
\[ \tau_n(i_T) : \tau_n(I_T) \to \tau_n(T) \]
is an isomorphism of $n$-categories.
\end{prop}

\begin{proof}
Let
\[ T = \Dn{i_1} \amalgd{i'_1} \dots \amalgd{i'_{m-1}} \Dn{i_m}. \]
It is immediate that, with the notation of paragraphs
\ref{paragr:def_table} and \ref{paragr:n_cat}, we have
\[ T = N_nL_n(G_T). \]
Using the fact that the functor $N_n$ is fully faithful, we obtain
\allowdisplaybreaks
\begin{align*}
\tau_n(T) 
  & \cong \tau_nN_nL_n(G_T) \\
  & \cong L_n(G_T) \\
  & \cong L_n\big(\Dn{i_1} \amalgd{i'_1} \dots \amalgd{i'_{m-1}}
  \Dn{i_m}\big) \\
  & \cong L_n(\Dn{i_1}) \amalg_{L_n(\Dn{i'_1})} \dots \amalg_{L_n(\Dn{i'_{m-1}})} L_n(\Dn{i_m}) \\
  & \cong \Dn{i_1} \amalgd{i'_1} \dots \amalgd{i'_{m-1}} \Dn{i_m} \\
  & \cong \tau_n(\Dn{i_1}) \amalg_{\tau_n(\Dn{i'_1})} \dots \amalg_{\tau_n(\Dn{i'_{m-1}})} \tau_n(\Dn{i_m}) \\
  & \cong \tau_n\big(\Dn{i_1} \amalgd{i'_1} \dots \amalgd{i'_{m-1}}
  \Dn{i_m}\big) \\
  & \cong \tau_n(I_T),
\end{align*}
thereby proving the result.
\end{proof}

\begin{lemma}\label{lemma:bij_Dz}
Let $T$ be an object of $\Theta_n$ different from $\Dn{0}$. Then the
map $\mbord{T} : \bord{T} \to T$ induces a bijection $\mbord{T, \Dn{0}} :
\bord{T}_{\Dn{0}} \to T_{\Dn{0}}$.
\end{lemma}

\begin{proof}
By definition, the morphism $\mbord{T}$ is a monomorphism. The map
$\mbord{T, \Dn{0}}$ is thus injective. Let us show that it is also surjective.
Let $x : \Dn{0} \to T$ be a morphism of~$\Theta_n$. Since $\Dn{0}$ is the
terminal object of $\Theta_n$, the morphism $x$ is a monomorphism.
The object~$T$ being different from $\Dn{0}$, the morphism $x$ is not an
identity and it thus factors through~$\bord{T}$ by definition of $\bord{T}$.
The map $\mbord{T, \Dn{0}}$ is thus surjective.
\end{proof}

\begin{prop}
Let $X$ be a presheaf on $\Theta_n$. Then the set of objects of
$\tau_n(X)$ is in canonical bijection with $X_{\Dn{0}}$.
\end{prop}

\begin{proof}
We have seen in paragraph \ref{paragr:trunc} that
the square
\[
\xymatrix{
  \pref{\Delta} \ar[r]^{t^*} & \pref{\Theta_n} \\
  \Cat \ar[u]^{N_1} \ar[r]_i & \ncat \ar[u]_{N_n}
}
\]
is commutative. By taking left adjoints, we obtain that the square
\[
\xymatrix{
  \pref{\Delta} \ar[d]_{\tau_1} &
  \pref{\Theta_n} \ar[l]_{t_!} \ar[d]^{\tau_n} \\
  \Cat & \ncat \ar[l]^t
}
\]
is also commutative (up to isomorphism). We thus have
\[
\begin{split}
\Ob(\tau_n(X)) = \Ob(t\tau_n(X)) \cong \Ob(\tau_1t_!(X)) = t_!(X)_0,
\end{split}
\]
where the last equality comes from the case $n = 1$ of the proposition which
is well-known (see for instance the proposition page 33 of \cite{GZ}). 
By the theory of Kan extensions, we have
\[ 
  t_!(X)_0 \cong
  \varinjlim_{(S, \Delta_{0} \to t(S)) \in (\Delta_{0}\bs \Theta_n)^\op} X_S,
\]
where $\Delta_{0}\bs \Theta_n$ denotes the category whose objects are pairs $(S,
\Delta_{0} \to t(S))$ consisting of an object $S$ of $\Theta_n$ and
a morphism $\Delta_{0} \to t(S)$ of $\Delta$, and whose morphisms are the obvious ones.
It follows from the canonical bijection
\[ \Hom_{\Delta}(\Delta_{0}, t(S)) \cong \Hom_{\Theta_n}(\Dn{0}, S) \]
that the category $\Delta_{0}\bs\Theta_n$ admits $(\Dn{0}, \id{\Delta_{0}})$ as an
initial object. The above colimit is thus canonically isomorphic to
$X_{\Dn{0}}$, thereby proving the result.
\end{proof}

\begin{coro}\label{coro:bij_ob}
Let $T$ be an object of $\Theta_n$ different from $\Dn{0}$. Then
the $n$-functor
\[ \tau_n(\mbord{T}) : \tau_n(\bord{T}) \to \tau_n(T) \]
is bijective on objects.
\end{coro}

\begin{proof}
This follows from Lemma \ref{lemma:bij_Dz} and the above proposition.
\end{proof}

\begin{paragr}
Let $u : C \to D$ be a strict $n$-functor. We will say that
$u$ is \ndef{fully faithful} if for every pair~$x, y$ of
objects of $C$, the $(n-1)$-functor 
\[ u_{x, y} : \Map_C(x, y) \to \Map_D(u(x), u(y)) \]
is an isomorphism of strict $(n - 1)$-categories. Explicitly, this means
that for every $k$ such that~$0 < k \le n$ and every pair $f, g$ of parallel
$(k - 1)$-arrows of~$C$, the $n$-functor $u$ induces a bijection
\[ \Hom_C(f, g) \to \Hom_D(u(f), u(g)). \]
\end{paragr}

\begin{lemma}\label{lemma:J_ff}
Let $C$ be a strict $n$-category and let $\epsilon = 0, 1$. Then the $n$-functor
\[ \Homi(\canseg{\epsilon}, C) :  \Homi(J, C) \to \Homi(\{\epsilon\}, C)
  \cong C \]
is fully faithful.
\end{lemma}

\begin{proof}
We will prove the case $\epsilon = 1$. The case $\epsilon = 0$ will follow
by duality. Recall that we denote by $\comp_0$ the composition in dimension
$0$ of two $k$-arrows of an $n$-category.

By definition, objects of $\Homi(J, C)$ are invertible $1$-arrows of $C$.
Let $f : x \to y$ and~$f' : x' \to y'$ be two invertible $1$-arrows of $C$.
A morphism from $f$ to $f'$ in $\Homi(J, C)$ is given by a pair $h : x \to
x', k : y \to y'$ of morphisms of $C$ making the square
\[
\xymatrix{
  x \ar[d]_f \ar[r]^h & x' \ar[d]^{f'} \\
  y \ar[r]_{k} & y'
}
\]
commute. Since $f$ and $f'$ are invertible, the pair $(h, k)$ is uniquely
determined by~$k$. This exactly means that
the map
\[ \Hom_{\Homi(J, C)}(f, f') \to \Hom_C(y, y') \]
induced by $u$ is a bijection. 

Let us now describe the $k$-arrows of $\Homi(J, C)$ for $k$ such that $0 < k
\le n$. By definition, a $k$-arrow of $\Homi(J, C)$ is given by a strict
$k$-functor $J \times \Dn{k} \to C$. Such a $k$-functor is given by a
tuple $(f, g, \alpha, \beta)$ where $f, g$ are two invertible $1$-arrows of
$C$ and~$\alpha, \beta$ are two $k$-arrows of $C$ such that the compositions
$\beta \comp_0 f$ and $g \comp_0 \alpha$ make sense and are equal. Two
$k$-arrows $(f, g, \alpha, \beta)$ and $(f', g', \alpha', \beta')$ of
$\Homi(J, C)$ are parallel if and only if $f = f'$, $g = g'$,
$\alpha$~and~$\alpha'$ are parallel, and $\beta$ and $\beta'$ are parallel.

Suppose now $k$ is such that $1 < k \le n$ and let $(f, g, \alpha, \beta)$
and $(f, g, \alpha', \beta')$ be two parallel $k$-arrows of $\Homi(J, C)$.
We have to show that the map
\[
  \Hom_{\Homi(J, C)}((f, g, \alpha, \beta), (f, g, \alpha', \beta')) \to
  \Hom_C(\beta, \beta')
\]
induced by $u$ is a bijection. The same argument as in dimension $0$
applies. Formally, an inverse of this map is given by
\[ \Gamma : \beta \to \beta' \mapsto (f, g, g^{-1} \comp_0 \Gamma \comp_0 f
  ,\Gamma). \qedhere \]
\end{proof}

\begin{prop}
Let $v$ be a fully faithful strict $n$-functor. Then $v$ has the unique
right lifting property with respect to strict $n$-functors bijective on
objects.
\end{prop}

\begin{proof}
The assertion is a special case of a standard result in enriched category
theory. Let us prove our particular case. We will use the notation of
paragraph \ref{paragr:n_cat}. Consider a commutative square
\[
\xymatrix{
  A \ar[d]_{u} \ar[r]^f & C \ar[d]^v \\
  B \ar[r]_g & D \pbox{,}
}
\]
where $u$ is a strict $n$-functor bijective on objects and $v$ is a fully
faithful strict $n$-functor. Let $h : B \to C$ be a lift. The
condition $hu = f$ imposes that if $x$ is an object of $B$, then~$h(x) =
f(u_0^{-1}(x))$. On the other hand, the condition $hv = g$ imposes that if
$x, y$ is a pair of objects of $B$, then $h_{x, y} = v_{fu_0^{-1}(x),
fu_0^{-1}(y)}^{-1}g$. The strict $n$-functor $h$ is thus unique. One
immediately checks that the formulas given above define a strict
$n$-functor, thereby proving the result.
\end{proof}

\begin{coro}\label{coro:bij_ob_iso}
Let $u : C \to D$ be a strict $n$-functor bijective on objects. Then the strict
$n$\nbd-functors
\[
u \times' \cansegd{\epsilon}{J} :  C \times J \amalg_{C \times
\{\epsilon\}} D \times \{\epsilon\} \to D \times J,\quad 
   \epsilon = 0, 1,
\]
are isomorphisms of strict $n$-categories.
\end{coro}

\begin{proof}
Let $A$ be a strict $n$-category. By the Yoneda lemma, it suffices to prove that  
\[
 \Hom_{\ncat}(D \times J, A) \to  \Hom_{\ncat}(C \times J \amalg_{C \times
   \{\epsilon\}} D \times \{\epsilon\}, A) 
\]
is a bijection, or, in other words, that the functor of the statement has
the unique left lifting property with respect to the unique strict
$n$-functor $A \to \Dn{0}$. By adjunction, this is equivalent to saying that
the functor 
\[ \Homi(\canseg{\epsilon}, A) :  \Homi(J, A) \to \Homi(\{\epsilon\}, A)
  \cong A \]
has the unique right lifting property with respect to the functor $C \to D$. This
follows from Lemma~\ref{lemma:J_ff} and the above proposition.
\end{proof}

\begin{paragr}
Let $u : C \to D$ be a functor. Recall that $u$ is said to be an
\ndef{iso-fibration} if for every invertible arrow $f' : x' \to y'$ of $D$
and every object $y$ of $C$ such that $u(y) = y'$, there exists an
invertible arrow $f : x \to y$ of $C$ such that $u(f) = f'$. In other words,
the functor~$u$ is an iso-fibration if it has the right lifting property
with respect to $\canseg{1} : \Dn{0} \to J$. Note that $u$ is an
iso-fibration if and only if $u^\op : C^\op \to D^\op$ is an iso-fibration,
that is, if and only if $u$ has the right lifting property with respect to
$\canseg{0} : \Dn{0} \to J$.
\end{paragr}

\begin{prop}\label{prop:W_fib}
Let $u$ be a strict $n$-functor and denote by~$\W$ the $\Theta_n$-localizer
generated by~$\clI_n$. Then the morphism $N_n(u)$ is a $\W$-fibration if and
only if the right truncation $t_r(u)$ of $u$ is an iso-fibration. In
particular, for any strict $n$-category $C$, its nerve $N_n(C)$ is
$\W$-fibrant.
\end{prop}

\begin{proof}
In this proof, by ``naive fibration'' we will mean ``naive $(\clI_n,
N_n(J))$-fibration''. (By paragraph~\ref{paragr:J}, $N_n(J)$ is an
injective separating interval and we will thus be able to apply
Theorem~\ref{thm:naive_fib}.)

By paragraph \ref{paragr:an_ext}, the morphism $N_n(u)$ is a naive fibration
if and only if it has the right lifting property with respect to
\[
  \Lambda^\infty_{N_n(J)}(\clI_n) \cup 
  \{\bord{T} \times N_n(J) \cup T \times \{\epsilon\} \to T \times N_n(J);
    \,\, T \in \Ob(\Theta_n),\,\, \epsilon = 0,1\}.
\]
By adjunction and using the fact that $\tau_n$ commutes with colimits and
binary products, we obtain that $N_n(u)$ is a naive fibration if and only
if $u$ has the right lifting property with respect to
\[
  \Lambda^\infty_J(\tau_n(\clI_n)) \cup 
  \{\tau_n(\bord{T}) \times J \amalg_{\tau_n(\bord{T}) \times \{\epsilon\}}
  \tau_n(T) \times \{\epsilon\} \to \tau_n(T) \times J;
    \,\, T \in \Ob(\Theta_n),\,\, \epsilon = 0,1\}.
\]
(A priori, we have only defined the $\Lambda^\infty$ construction in
presheaf categories. Nevertheless, it is clear that the definition still makes
sense in any category admitting finite products and pushouts.) By
Proposition~\ref{prop:iso_tau_i_T}, the $n$-functors of $\tau_n(\clI_n)$ are
isomorphisms. Since isomorphisms are stable under pushout and binary
product, it follows that the $n$-functors of
$\Lambda^\infty_J(\tau_n(\clI_n))$ are also isomorphisms. Moreover, by
Corollaries \ref{coro:bij_ob} and \ref{coro:bij_ob_iso}, the
$n$\nbd-functors of
\[
  \{\tau_n(\bord{T}) \times J \amalg_{\tau_n(\bord{T}) \times \{\epsilon\}}
  \tau_n(T) \times \{\epsilon\} \to \tau_n(T) \times J; \,\, T \in
  \Ob(\Theta_n)\bs\{\Dn{0}\},\,\, \,\, \epsilon = 0,1\} 
\]
are again isomorphisms. It follows that $N_n(u)$ is a naive fibration
if and only if~$u$~has the right lifting property with respect to
\[
  \tau_n(\bord{\Dn{0}}) \times J \amalg_{\tau_n(\bord{\Dn{0}}) \times \{\epsilon\}}
  \tau_n(\Dn{0}) \times \{\epsilon\} \to \tau_n(\Dn{0}) \times J.
\]
But this $n$-functor is nothing but $\canseg{\epsilon} : \Dn{0} \to J$ and
$N_n(u)$ is thus a naive fibration if and only if $t_r(u)$ is an
iso-fibration. In particular, every strict $n$-category $C$ is $(\clI_n,
N_n(J))$\nbd-fibrant and hence $\W$-fibrant by Theorem~\ref{thm:naive_fib}.
The result follows from the same theorem since $N_n(u)$ is a morphism
between $\W$-fibrant objects.
\end{proof}

\begin{rem}
Let $\W$ be the $\Theta_n$-localizer generated by $\clI_n$. By the above
proposition, nerves of strict $n$-categories are $\W$-fibrant objects. We
will see in the introduction of the next section that this implies that
when $n > 1$ the functor $N_n$ cannot have the property that a strict
$n$-functor $u$ is an equivalence of strict $n$-categories if and only if
$N_n(u)$ is a $\W$-equivalence. This shows that we have to add other
generators to obtain a model for $(\infty, n)$-categories.
\end{rem}

\begin{prop}
Let $k$ be an integer such that $1 < k \le n$. Then the morphism
\[ N_n(j_k) : N_n(J_k) \to N_n(\Dn{k-1}) \]
is not a trivial fibration of $\pref{\Theta_n}$.
\end{prop}

\begin{proof}
Let $T$ be the object $\Dn{k-1} \amalg_{\Dn{k-2}} \Dn{k-1}$ of $\Theta_n$. We
claim that $N_n(j_k)$ does not have the right lifting property with respect
to $\mbord{T} : \bord{T} \to T$, or, by adjunction, that $j_k$~does not have
the right lifting property with respect to $\tau_n(\mbord{T}) :
\tau_n(\bord{T}) \to \tau_n(T)$. Since $\tau_n$~commutes with colimits, we
have
\[
\tau_n(\bord{T})
= \varinjlim_{S \stackrel{\neq}{\hookrightarrow} T} S,
\]
the colimit being taken over the category of non-trivial monomorphisms of
$\Theta_n$ whose target is $T$. Using this formula, one easily
obtains that a strict $n$-functor from $\tau_n(\bord{T})$ to a strict $n$-category $C$ is
given by a triangle of $(k-1)$-arrows composable in dimension~$k-2$. It
follows that a strict $n$-functor $u : C \to D$ has the right lifting property with
respect to~$\tau_n(\mbord{T}) : \tau_n(\bord{T}) \to \tau_n(T)$ if and only
if every such triangle in $C$ which is sent to a commutative triangle in $D$ is
already commutative in $C$, or equivalently, if and only if $u$~is injective
on parallel $(k-1)$-arrows, i.e., if two parallel $(k-1)$-arrows $f, g$ of $C$ are
sent to the same $(k-1)$-arrow in $D$, then $f = g$. This condition is not
fulfilled by $j_k$, thereby proving that $N_n(j_k)$ is not a trivial fibration.
\end{proof}

\begin{coro}
Let $k$ be an integer such that $1 < k \le n$. Then the morphism 
\[ N_n(j_k) : N_n(J_k) \to N_n(\Dn{k-1}) \]
is not in the $\Theta_n$-localizer generated by $\clI_n$.
\end{coro}

\begin{proof}
Denote by $\W$ the $\Theta_n$-localizer generated by $\clI_n$.
The functor $t_r(j_k)$ is an iso-fibration and the morphism $N_n(j_k)$ is
thus a $\W$-fibration by Proposition \ref{prop:W_fib}. By the above
proposition, this morphism is not a trivial fibration. It follows that it is
not a $\W$-equivalence.
\end{proof}

\section{Nerves of strict $n$-categories and $n$-quasi-categories}
\label{sec:7}

If $C$ is a category, it is well-known that its nerve $N_1(C)$ is a
quasi-category. In this section, we show that it is not true that nerves of
strict $n$-categories are $n$-quasi-categories as soon as $n > 1$. More
precisely, we will show that the nerve $N_n(C)$ of a strict
$n$\nbd-category~$C$ is an $n$-quasi-category if and only if $C$ has no
non-trivial invertible $k$-arrows for $k > 1$.

\medbreak

There is actually an abstract reason for which nerves of strict
$n$\nbd-categories cannot be $n$-quasi-categories. Indeed, if they were
$n$-quasi-categories, they would be cofibrant-fibrant objects in the model
category of $n$-quasi-categories. That would imply that any weak equivalence
between nerves of strict $n$\nbd-categories has an inverse up to homotopy.
Assuming that $N_n$ has the property that a strict $n$-functor $u$ is an
equivalence of strict $n$-categories if and only if $N_n(u)$ is a weak
equivalence of $n$\nbd-quasi-categories (which is expected but not proved in
this paper) and using the fully faithfulness of $N_n$, we would get that if
a strict $n$-functor $C \to D$ is an equivalence of strict $n$-categories,
then there exists a strict $n$-functor $D \to C$ which is an equivalence of
strict $n$-categories, which is false when $n > 1$.

\medbreak

Throughout the section, we fix an integer $n \ge 1$ and an integer $k$ such
that $1 < k \le n$.

\begin{paragr}
We will say that a strict $n$-category $C$ is \ndef{rigid in dimension $k$}
if any (strictly) invertible $k$-arrow of $C$ is the identity of a
$(k-1)$-arrow.
\end{paragr}

\begin{prop}\label{prop:rig_dim_k}
A strict $n$-category $C$ is rigid in dimension $k$ if and only if the unique
$n$-functor $C \to \Dn{0}$ has the right lifting property with respect to
the $n$-functor
\[
s^0_k \times' \mbordseg{J} :
\Dn{k-1} \times J \amalg_{\Dn{k-1} \times \bord{J}} J_k \times \bord{J} \to
J_k \times J
\]
induced by $s^0_k : \Dn{k-1} \to J_k$ and $\mbordseg{J} : \bordseg{J} \to
J$.
\end{prop}

\begin{proof}
Let us denote by $i$ the unique arrow of $J$ from $0$ to $1$, and by
$\gamma$ the unique $k$-arrow of $J_k$ whose source (resp.~whose target)
corresponds to the morphism $s^0_k : \Dn{k-1} \to J_k$ (resp.~to the
morphism $s^1_k : \Dn{k-1} \to J_k$).

To any strict $n$-functor
\[ u : J_k \times J \to C, \]
we can associate a tuple $(f, g, \alpha, \beta)$ given by
\[
f = u(0, i),\quad g = u(1, i),\quad \alpha = u(\gamma, 0)\quad\text{and}\quad
\beta = u(\gamma, 1).
\]
Conversely, if $(f, g, \alpha, \beta)$ is a tuple where $f, g$ are
invertible $1$-arrows of $C$ and $\alpha, \beta$ are invertible $k$-arrows of
$C$, then the above formulas define a strict $n$-functor if and only if we
have
\[ g \comp_0 \alpha = \beta \comp_0 f. \]
Similarly, to any strict $n$-functor
\[
u : \Dn{k-1} \times J \amalg_{\Dn{k-1} \times
\bord{J}} J_k \times \bord{J} \to C,
\]
we can associate a tuple $(f, g, \alpha, \beta)$ given by
\[
f = u(0, i),\quad g = u(1, i),\quad \alpha = u(\gamma, 0)
\quad\text{and}\quad \beta = u(\gamma,1).
\]
Conversely, if $(f, g, \alpha, \beta)$ is a tuple where $f, g$ are
invertible $1$-arrows of $C$ and $\alpha, \beta$ are invertible $k$-arrows of
$C$, then the above formulas define a strict $n$-functor if and only if we
have
\[ g \comp_0 \Gls{k}(\alpha) = \Gls{k}(\beta) \comp_0 f. \]
Moreover, if a tuple $(f, g, \alpha, \beta)$ defines a functor $J_k \times J \to C$,
then the induced functor $\Dn{k-1} \times J \amalg_{\Dn{k-1} \times
\bord{J}} J_k \times \bord{J} \to C$ is also defined by $(f, g, \alpha,
\beta)$.

The above analysis shows that $C \to \Dn{0}$ has the right lifting property
with respect to~$s_k^0 \times' \mbordseg{J}$ if and only if for every tuple
$(f, g, \alpha, \beta)$ where $f, g$ are invertible $1$-arrows of~$C$ and
$\alpha, \beta$ are invertible $k$-arrows of $C$, if 
\[ g \comp_0 \Gls{k}(\alpha) = \Gls{k}(\beta) \comp_0 f, \]
then
\[ g \comp_0 \alpha = \beta \comp_0 f. \]
This clearly holds if $\alpha$ and $\beta$ are identities of $(k-1)$-arrows
and it hence holds for any~$\alpha$ and~$\beta$ if $C$ is rigid in dimension
$k$. Conversely, suppose $C$ is not rigid in dimension $k$. By definition,
there exists an invertible $k$-arrow $\alpha : v \to w$ which is not an
identity. Denote by $x$ (resp.~by $y$) the source (resp.~the target) of $v$
in dimension $0$. Consider the tuple $(\id{x}, \id{y}, \alpha, \id{v})$. We
indeed have
\[ \id{y} \comp_0 \Gls{k}(\alpha) = v = \Gls{k}(\id{v}) \comp_0 \id{x}, \]
but
\[ \id{y} \comp_0 \alpha = \alpha \neq \id{v} = \id{v} \comp_0 \id{x}, \]
thereby proving the result.
\end{proof}

\begin{rem}
The above proposition is false for $k = 1$: for any category $C$, the
$n$-functor $C \to \Dn{0}$ has the right lifting property with respect to
$s^0_1 \times' \mbordseg{J}$.
\end{rem}

\begin{rem}\label{rem:equiv_s0_s1}
We could have replaced $s^0_k$ by $s^1_k$ in the statement of the
proposition. This immediately follows from the existence of an automorphism
$\sigma$ of $J_k \times J$ satisfying
\[ s^1_k \times' \mbordseg{J} = \sigma \circ (s^0_k \times' \mbordseg{J}).
\]
\end{rem}

\begin{paragr}
We will say that a strict $n$-functor $u : C \to D$ is \ndef{rigid in
dimension $k$} if it has the right lifting property with respect to the
$n$-functor
\[
s^0_k \times' \mbordseg{J} :
\Dn{k-1} \times J \amalg_{\Dn{k-1} \times \bord{J}} J_k \times \bord{J} \to
J_k \times J.
\]
(By the above remark, we could replace $s^0_k$ by $s^1_k$.)
The above proposition can then be rephrased by saying that a strict
$n$-category $C$ is rigid in dimension $k$ if and only if the
$n$-functor $C \to \Dn{0}$ is rigid in dimension $k$.
\end{paragr}

\begin{lemma}\label{lemma:box_epi}
The $n$-functor
\[
s^0_k \times' \mbordseg{J} :
\Dn{k-1} \times J \amalg_{\Dn{k-1} \times \bord{J}} J_k \times \bord{J} \to
J_k \times J
\]
is an epimorphism.
\end{lemma}

\begin{proof}
We have to show that for any strict $n$-category $C$, the map
$\Hom_{\ncat}(s^0_k \times' \mbordseg{J}, C)$ is injective. This is an
immediate consequence of the description of this map given in the proof
of Proposition~\ref{prop:rig_dim_k}.
\end{proof}

\begin{prop}\label{prop:rig_source}
Let $u : C \to D$ be a strict $n$-functor. If the $n$-category $C$ is rigid
in dimension $k$, then so is the $n$-functor $u$.
\end{prop}

\begin{proof}
If follows formally from the fact that the $n$-functor $s^0_k \times'
\mbordseg{J}$ is an epimorphism (see the above lemma) that if $C \to \Dn{0}$
has the right lifting property with respect to this $n$\nbd-functor, so does
any strict $n$-functor whose source is $C$. The result thus follows from
Proposition~\ref{prop:rig_dim_k}.
\end{proof}

\begin{prop}\label{prop:rig_stable}
Let $C$ be a strict $n$-category. If $C$ is rigid in dimension $k$, then so
is the $n$-category $\Homi(J, C)$.
\end{prop}

\begin{proof}
Invertible $k$-arrows of $\Homi(J, C)$ are in canonical bijection with
strict $n$-functors $J_k \times J \to C$. We saw in the proof of
Proposition~\ref{prop:rig_dim_k} that these $n$-functors are in canonical
bijection with tuples $(f, g, \alpha, \beta)$, where $f, g$ are invertible
$1$-arrows of $C$ and $\alpha, \beta$ are invertible $k$-arrows of $C$
satisfying $g \comp_0 \alpha = \beta \comp_0 f$. In this correspondence,
such a tuple is an identity in the $n$-category $\Homi(J, C)$ if and only if
$\alpha$ and $\beta$ are identities in~$C$. The result follows immediately.
\end{proof}

\begin{rem}
The previous proposition is easily seen to hold when $J$ is replaced by any
strict $n$-category $D$: all one has to do to adapt the proof is to
describe the strict $n$-functors $J_k \times D \to C$.
\end{rem}

\begin{prop}
Let $C$ be a strict $n$-category. Then $N_n(C)$ is an $n$-quasi-category if
and only if $C$ is rigid in dimension $k$ for $1 < k \le n$.
\end{prop}

\begin{proof}
Since $N_n(J)$ is an injective separating interval (see
paragraph~\ref{paragr:J}), by Theorem~\ref{thm:naive_fib}, 
$N_n(C)$ is an $n$-quasi-category if and only if it is $(\clI_n \cup
\clJ'_n, N_n(J))$\nbd-fibrant. By Proposition~\ref{prop:W_fib},  
$N_n(C)$ is always $(\clI_n, N_n(J))$-fibrant. We thus have to understand
when the unique map $N_n(C) \to \Dn{0}$ has the right lifting property with
respect to $\Lambda^\infty_{N_n(J)}(\clJ'_n)$. By using the same arguments
(and notation) than in the proof of Proposition~\ref{prop:W_fib}, this
condition is equivalent to the fact that the map $C \to \Dn{0}$ has
the right lifting property with respect to
\[ \Lambda^\infty_{J}(\tau_n(\clJ'_n)) = 
\bigcup_{1 < k \le n} \Lambda^\infty_{J}(\{s_k^\epsilon : \Dn{k-1} \to
J_k;\,\,\epsilon = 0, 1\}).
\]

We now fix $k$ such that $1 < k \le n$. We will show that $C \to \Dn{0}$ has
the right lifting property with respect to $\Lambda^\infty_{J}(\{s_k^0,
s_k^1\})$ if and only if $C$ is rigid in dimension $k$. 

Given a class $S$ of strict $n$-functors, define $V^\infty_J(S)$ to be the
smallest class of strict $n$-functors containing $S$ and closed under the
operation 
\[ f : A \to B \quad \longmapsto \quad \Homi'(\mbordseg{J}, f) : \Homi(J, A) \to
\Homi(\bordseg{J}, A) \times_{\Homi(\bordseg{J}, B)} \Homi(J, B).
\]
By adjunction (see paragraph~\ref{paragr:box_cart_clos}), if $S$ and $T$ are
two classes of strict $n$-functors, we have
\[
\Lambda_J^\infty(S) \orth T
\quad\Leftrightarrow\quad
S \orth V_J^\infty(T).
\]
In particular, we have
\[
\Lambda_J^\infty(\{s^0_k, s^1_k\}) \orth \{p\}
\quad\Leftrightarrow\quad
\{s^0_k, s^1_k\} \orth V_J^\infty(\{p\}),
\]
where $p$ denotes the unique $n$-functor $C \to \Dn{0}$. Since the
$n$-functors $s^0_k$ and $s^1_k$ both admit a retraction, they have the left
lifting property with respect to $p : C \to \Dn{0}$. By adjunction, we thus
have
\[
\{s^0_k, s^1_k\} \orth V_J^\infty(\{p\})
\quad\Leftrightarrow\quad
\{s^0_k \times' \mbordseg{J}, s^1_k \times' \mbordseg{J}\} \orth V_J^\infty(\{p\}).
\]
This last condition can be rephrased by saying that the $n$-functors of
$V_J^\infty(\{C \to D_0\})$ are rigid in dimension $k$. In particular, this
condition implies that $C$ is rigid in dimension~$k$. Conversely, if $C$ is
rigid in dimension $k$, then by Proposition~\ref{prop:rig_stable}, the
sources of the $n$-functors of $V_J^\infty(\{C \to D_0\})$ are rigid in
dimension $k$ and the result follows by Proposition~\ref{prop:rig_source}.
\end{proof}

\begin{coro}\label{coro:J_k_not_fib}
For every $k$ such that $1 < k \le n$, the presheaf $N_n(J_k)$ is not an
$n$-quasi-category.
\end{coro}

\begin{proof}
This follows immediately from the proposition since $J_k$ has 
non-trivial invertible $k$-arrows.
\end{proof}

\begin{coro}
If $C$ is a category, then $N_n(C)$ is an $n$-quasi-category. 
\end{coro}

\begin{proof}
This follows immediately from the proposition since categories have only 
trivial $k$-arrows for $k > 1$.
\end{proof}

\section{Comparison with Rezk
\texorpdfstring{$\Theta_n$}{Theta\_n}-spaces}\label{sec:8}

In \cite{RezkThSp}, Rezk introduced a notion of $\Theta_n$-spaces as fibrant
objects of a model category structure on simplicial presheaves on
$\Theta_n$. Since the cofibrations of this model category structure are the
monomorphisms, the weak equivalences define a $(\Theta_n \times
\Delta)$-localizer. The purpose of this section is to show that this
localizer is the simplicial completion of the localizer of
$n$-quasi-categories. It will then follow formally from the theory of
simplicial completion that we have two Quillen equivalences between the
corresponding model category structures.

\begin{paragr}
The \ndef{localizer of Rezk $\Theta_n$-spaces} is
the $(\Theta_n \times \Delta)$-localizer generated by 
\[
 \Wvert \cup
 p^*(\clI_n) \cup
 \{p^*(N_n(j) : N_n(J) \to N_n(\Dn{0}))\} \cup
 p^*(\clJ_n),
\]
where $\Wvert$ is the class of vertical equivalences defined in
paragraph \ref{paragr:def_simp_comp} (with $A = \Theta_n$).
We will denote it by $\WRezk$. By the 2-out-of-3 property, this localizer is
also generated by
\[
 \Wvert \cup
 p^*(\clI_n) \cup
 \{p^*(N_n(\canseg{\epsilon}) : N_n(\Dn{0}) \to N_n(J));
 \,\, \epsilon = 0, 1\}) \cup
 p^*(\clJ'_n).
\]
Since the localizer $\Wvert$ is accessible, so is the localizer $\WRezk$. It
follows from the connection between localizers and Bousfield localization
(see Proposition~\ref{prop:loc_Bous_loc}) that the $\WRezk$-model category
structure on $\pref{\Theta_n \times \Delta}$ is nothing but the model
category structure~$\Theta_n\mathrm{Sp}_\infty$ introduced by Rezk in
\cite{RezkThSp}. The fibrant objects of this model category are
called $(\infty, n)$-$\Theta$-spaces by Rezk. We will call them \ndef{Rezk
$\Theta_n$-spaces} and the $\WRezk$-model category will be called \ndef{the
model category of Rezk $\Theta_n$-spaces}.
\end{paragr}

\begin{thm}[Rezk]\label{thm:Rezk}
The model category of Rezk $\Theta_n$-spaces is simplicial, left proper and
cartesian.
\end{thm}

\begin{proof}
The model category of Rezk $\Theta_n$-spaces is a left Bousfield
localization of the vertical model category. By Theorem~\ref{thm:vert_mcs},
the vertical model category is simplicial. It follows from Theorem~4.11
of~\cite{HirschMC} that the model category of Rezk $\Theta_n$-spaces is
simplicial.

By definition, every object is cofibrant in the model category of Rezk
$\Theta_n$-spaces. This structure is hence left proper.

The fact that the model category of Rezk $\Theta_n$-spaces is cartesian is
highly non-trivial. It is a special case of one of the main results of
\cite{RezkThSp} (Proposition 11.5).
\end{proof}

\begin{thm}\label{thm:Rezk_simpl_compl}
The $(\Theta_n \times \Delta)$-localizer of Rezk $\Theta_n$-spaces is the
simplicial completion of the $\Theta_n$-localizer of $n$-quasi-categories.
\end{thm}

\begin{proof}
By Theorem \ref{thm:ThReedy}, the category $\Theta_n$ is a regular skeletal
Reedy category. We can thus apply Theorem \ref{thm:fRezk_simp_comp} to the
set $S =\clI_n \cup \clJ_n$ and to the injective separating
interval~$N_n(J)$ (see paragraph \ref{paragr:J}). We obtain that the
simplicial completion of $\WQCat$ is the $(\Theta_n \times \Delta)$-localizer
$\WfRezk$ defined in paragraph~\ref{paragr:loc_fRezk}. By definition, this
localizer is generated by
\[
\Wvert \cup 
p^*(\clI_n \cup \clJ_n) \cup 
\{p^*(\cansegd{\epsilon}{X} : X \to X \times N_n(J));
 \,\, X \in \Ob(\pref{\Theta_n}),\,\, \epsilon = 0,1\}.
\]
Since $\canseg{\epsilon} = \cansegd{\epsilon}{\Dn{0}}$, the localizer
$\WRezk$ is contained in $\WfRezk$. Conversely, we have to show that
$\cansegd{\epsilon}{X} = X \times \canseg{\epsilon}$ belongs to $\WRezk$. But by
Theorem~\ref{thm:Rezk}, the $\WRezk$-model category is cartesian and the
localizer $\WRezk$ is hence closed under binary product, thereby proving the
result.
\end{proof}

\begin{thm}\label{thm:Q_equiv_Rezk}
Let us endow $\pref{\Theta_n}$ (resp.~$\pref{\Theta_n \times \Delta}$) with
the model category structure of $n$-quasi-categories (resp.~with the model
category structure of Rezk $\Theta_n$-spaces).
\begin{enumerate}
\item 
Then the adjunction
\[
  p^* : \pref{\Theta_n} \rightleftarrows \pref{\Theta_n \times \Delta} : i^*_0
\]
is a Quillen equivalence. 
\item
Let $D : \Delta \to \pref{\Theta_n}$ be a cosimplicial $\WQCat$-resolution. Then the
adjunction
\[
\Real_D : \pref{\Theta_n \times \Delta} \rightleftarrows \pref{\Theta_n} : \Sing_D
\]
is a Quillen equivalence.
\end{enumerate}
\end{thm}

\begin{proof}
This follows from Propositions \ref{prop:Qeq_pA} and \ref{prop:Qeq_RealD},
and from the above theorem.
\end{proof}

\begin{coro}
The model category of $n$-quasi-categories is cartesian closed.
\end{coro}

\begin{proof}
By Theorems \ref{thm:Rezk} and \ref{thm:Rezk_simpl_compl}, the simplicial
completion of the localizer of $n$-quasi-categories is cartesian. It follows
from Corollary \ref{coro:loc_cart_simpl} that this localizer is cartesian,
and from Proposition \ref{prop:loc_cart_eq} that the associated model
category is cartesian closed.
\end{proof}

\begin{paragr}
Let $\Deltat{\bullet} : \Delta \to \Cat$ be the functor
\[ \Delta \hookrightarrow \Cat \xrightarrow{\Pi_1} \Gpd \hookrightarrow \Cat, \]
where $\Gpd$ denotes the category of groupoids and $\Pi_1$ is the
fundamental groupoid functor, i.e., the left adjoint to the inclusion
functor $\Gpd \to \Cat$. In other words, the
functor~$\Deltat{\bullet}$ sends $\Delta_k$ to $\Deltat{k}$, the simply
connected groupoid with objects $\{0, \dots, k\}$.
\end{paragr}

\begin{prop}
The functor $N_n\Deltat{\bullet} : \Delta \to \pref{\Theta_n}$ is a
simplicial $\WQCat$-resolution of $\pref{\Theta_n}$.
\end{prop}

\begin{proof}
It is clear that $N_n\Deltat{0} \amalg N_n\Deltat{0} \to N_n\Deltat{1}$ is a
monomorphism. Let $k \ge 0$. We have to prove that for every presheaf $X$ on
$\Theta_n$, the projection $X \times N_n\Deltat{k} \to X$ is a
$\WQCat$-equivalence. It suffices to prove that the unique morphism
$N_n\Deltat{k} \to \Dn{0}$ is a trivial fibration. But this morphism can be
obtained by applying the functor $N_n$ to the unique functor $\Deltat{k} \to
\Delta_0$. The result thus follows from Corollary \ref{coro:fib_triv_fun}
since $\Deltat{k} \to \Delta_0$ is an equivalence of categories surjective
on objects between preorders.
\end{proof}

\begin{coro}
The simplicial $\WQCat$-resolution $N_n\Deltat{\bullet}$ induces a Quillen
equivalence
\[
  \Real_{N_n\Deltat{\bullet}} : \pref{\Theta_n \times \Delta}
  \rightleftarrows \pref{\Theta_n} : \Sing_{N_n\Deltat{\bullet}},
\]
where $\pref{\Theta_n \times \Delta}$ (resp.~$\pref{\Theta_n}$)
is endowed with the model category structure of Rezk $\Theta_n$-spaces
(resp.~with the model category structure of $n$-quasi-categories). 
\end{coro}

\begin{proof}
By the above proposition, the functor $N_n\Deltat{\bullet}$ is a simplicial
$\WQCat$-resolution. The result thus follows from
Theorem~\ref{thm:Q_equiv_Rezk} applied to $D = N_n\Deltat{\bullet}$.
\end{proof}

\begin{rem}
When $n = 1$, we recover from Theorem \ref{thm:Q_equiv_Rezk} and the above
corollary the two Quillen equivalences between quasi-categories and complete
Segal spaces defined by Joyal and Tierney in \cite{JoyalTierneyQCatSeg}.
\end{rem}

\appendix

\section{Localizers and Bousfield localization}\label{sec:A}

In this appendix, we compare the language of localizers to the language of
\hbox{Bousfield} localization. This comparison is needed for instance
to see that the definition of Rezk $\Theta_n$\nbd-spaces we gave in terms of
$(\Theta_n\times\Delta)$-localizers coincide with the original definition of
Rezk.

\begin{paragr}
In this appendix, the category $\pref{\Delta}$ of simplicial sets will
always be endowed with its classical model category structure (i.e., the one
defined by Quillen in \cite{Quillen}).

If $\M$ is a model category, the category $\Homi(\Delta^\op, \M)$ of
simplicial objects in $\M$ will always be endowed with the Reedy model
category structure. Similarly, the category $\Homi(\Delta, \M)$ of
cosimplicial objects in $\M$ will always be endowed with the Reedy model
category structure. The weak equivalences of these structures are the
objectwise weak equivalences. All we will need about their cofibrations is
that they only depend on the cofibrations of $\M$. We refer the reader to
Chapters~15 and~16 of \cite{HirschMC} for details on these model categories.

Recall that the functors
\[
\M \to \Homi(\Delta^\op, \M)
\quad\text{and}\quad
\M \to \Homi(\Delta, \M)
\]
sending an object of $\M$ to the associated constant simplicial object
(resp.~to the associated constant cosimplicial object) are fully faithful.
Moreover, they induce fully faithful functors
\[
\Ho(\M) \to \Ho(\Homi(\Delta^\op, \M))
\quad\text{and}\quad
\Ho(\M) \to \Ho(\Homi(\Delta, \M))
\]
between homotopy categories. In the sequel, we will often consider these
four functors as inclusions.
\end{paragr}

\begin{paragr}
Let $\M$ be a model category. The functor
\[ \Hom_\M : \M^\op \times \M \to \Set \]
induces a functor
\[ \Homi(\Delta, \M)^\op \times \Homi(\Delta^\op, \M) \to \pref{\Delta} \]
that we will denote in the same way. Explicitly, if $X$ is a cosimplicial
object in $\M$ and $Y$ is a simplicial object in $\M$, the simplicial set
$\Hom_\M(X, Y)$ is given by 
\[ \Hom_\M(X, Y)_n = \Hom_\M(X_n, Y_n), \quad n \ge 0. \]
It is well-known that this functor preserves weak equivalences between
fibrant objects and hence admits a right derived functor
\[ \RHom_\M : \Ho(\Homi(\Delta, \M))^\op \times \Ho(\Homi(\Delta^\op, \M)) \to
\Ho(\pref{\Delta}). \]
In particular, we obtain a functor
\[ \RHom_\M : \Ho(\M)^\op \times \Ho(\M) \to \Ho(\pref{\Delta}). \]
\end{paragr}

\begin{paragr}\label{paragr:comp_RHom}
Let $X$ and $Y$ be two objects of a model category $\M$. Denote by $p :
\pref{\Delta} \to \Ho(\pref{\Delta})$ the localization functor. The theory
of derived functors gives the formula
\[ \RHom_\M(X, Y) \cong p(\Hom_\M(\tilde{X}, \tilde{Y})), \]
where $\tilde{X}$ is any cofibrant cosimplicial replacement of $X$ and
$\tilde{Y}$ is any fibrant simplicial replacement of $Y$. It follows
from Proposition 17.4.6 of \cite{HirschMC} that $\RHom_\M(X, Y)$ can be
computed using the simpler formula
\[ \RHom_\M(X, Y) \cong p(\Hom_\M(\tilde{X}, Y')), \]
where $\tilde{X}$ is any cofibrant cosimplicial replacement of $X$ and $Y'$
is any fibrant replacement of~$Y$ in $\M$; or using the dual formula
\[ \RHom_\M(X, Y) \cong p(\Hom_\M(X', \tilde{Y})), \]
where $X'$ is any cofibrant replacement of $X$ in $\M$ and $\tilde{Y}$
is any fibrant simplicial replacement of~$Y$. Moreover, if the model category
$\M$ is simplicial, it follows from Proposition~16.6.4 of~\cite{HirschMC}
that 
\[ \RHom_\M : \Ho(\M)^\op \times \Ho(\M) \to \Ho(\pref{\Delta}) \]
is the right derived functor of the simplicial enrichment
\[ \Map : \M^\op \times \M \to \pref{\Delta} \]
of $\M$. In other words, we have
\[ \RHom_\M(X, Y) \cong p(\Map(X', Y')), \]
where $X'$ is any cofibrant replacement of $X$ in $\M$ and $Y'$ is any
fibrant replacement of~$Y$ in~$\M$.
\end{paragr}

\begin{lemma}\label{lemma:cmp_RHom}
Let $\M$ and $\M'$ be two model category structures on the same category
such that
\begin{itemize}
\item the class of cofibrations of $\M$ and $\M'$ are the same;
\item the class of weak equivalences of $\M$ is included in the class of
weak equivalences of $\M'$.
\end{itemize}
Then for every object $X$ and every fibrant object $Y$ of $\M'$, we have a
natural isomorphism
\[ \RHom_\M(X, Y) \cong \RHom_{\M'}(X, Y). \]
\end{lemma}

\begin{proof}
Note that the hypothesis on $\M$ and $\M'$ implies that
\begin{itemize}
\item every fibrant object of $\M'$ is fibrant in $\M$;
\item the class of cofibrations of $\Homi(\Delta, \M)$ and $\Homi(\Delta,
\M')$ are the same;
\item the class of weak equivalences of $\Homi(\Delta, \M)$ is included in
the class of weak equivalences of $\Homi(\Delta, \M')$.
\end{itemize}
In particular, by the first point, the object $Y$ is fibrant in $\M$.  Now
let $\tilde{X}$ be a cofibrant cosimplicial replacement of $X$ in
$\Homi(\Delta, \M)$. By the second and third points, $\tilde{X}$ is also a
cofibrant cosimplicial replacement of $X$ in $\Homi(\Delta, \M')$. It follows
by paragraph \ref{paragr:comp_RHom} that
\[ \Hom_\M(\tilde{X}, Y) = \Hom_{\M'}(\tilde{X}, Y) \]
can be used to compute
both $\RHom_{\M}(X, Y)$ and $\RHom_{\M'}(X, Y)$, thereby proving the result.
\end{proof}

\begin{paragr}
Let $\M$ be a model category, let $\clC$ be a class of morphisms of $\M$ and
let $\clD$ be a class of objects of $\M$.
\begin{itemize}
\item 
An object $Z$ of $\M$ is said to be \ndef{$\clC$-local} if for every
morphism $f : X \to Y$ in $\clC$, the morphism
\[ \RHom_{\M}(f, Z) : \RHom_{\M}(Y, Z) \to \RHom_{\M}(X, Z) \]
is an isomorphism.

\item
A morphism $f : X \to Y$ of $\M$ is said to be a \ndef{$\clD$-local equivalence}
if for every object $Z$ in $\clD$, the morphism
\[ \RHom_{\M}(f, Z) : \RHom_{\M}(Y, Z) \to \RHom_{\M}(X, Z) \]
is an isomorphism. If the class $\clD$ consists of a single object $Z$, we
will say that $f$ is a $Z$-local equivalence.

\item
A morphism $f : X \to Y$ of $\M$ is said to be a \ndef{$\clC$-local
equivalence} if it is an $\clL$-local equivalence, where $\clL$ is the class of
$\clC$-local objects.
\end{itemize}
\end{paragr}

\begin{rem}\label{rem:loc_equiv}
By definition, if $\clC$ is a class of morphisms of a model category $\M$,
then the elements of $\clC$ are $\clC$-local equivalences. By functoriality of
$\RHom_\M$, the weak equivalences of $\M$ are also $\clC$-local
equivalences and even $\Ob(\M)$-local equivalences. This property actually
characterizes weak equivalences:
\end{rem}

\begin{prop}\label{prop:we_M_loc}
The weak equivalences of a model category $\M$ are precisely the
$\Ob(\M)$-local equivalences. More generally, they are the $\clD$-local
equivalences, where $\clD$ is any class of objects of $\M$ such that every
object of $\M$ is weakly equivalent to an object in $\clD$.
\end{prop}

\begin{proof}
This is a direct consequence of the formula
\[ \pi_0(\RHom_\M(X, Y)) \cong \Hom_{\Ho(\M)}(X, Y). \]
(See Theorem 17.7.2 of \cite{HirschMC}.)
\end{proof}

\begin{lemma}\label{lemma:Loc_loc}
Let $A$ be a small category and let $\M$ be a model category structure on
$\pref{A}$ whose cofibrations are the monomorphisms. Let $\clD$ be a class
of objects of $\pref{A}$. Then the class of $\clD$-local equivalences is an
$A$-localizer. In particular, if $\clC$ is a class of morphisms
of~$\pref{A}$, then the class of $\clC$-local equivalences is an
$A$-localizer.
\end{lemma}

\begin{proof}
The 2-out-of-3 property is immediate by functoriality of $\RHom_\M$.

By hypothesis, the trivial fibrations of $\pref{A}$ are the trivial
fibrations of $\M$. In particular, they are weak equivalences in $\M$ and
hence $\clD$-local equivalences.

Let us now show the stability conditions of the class of morphisms which are
both monomorphisms and $\clD$-local equivalences. Consider a cocartesian square
\[
\xymatrix{
X \ar[d]_f \ar[r] & X' \ar[d]^{f'} \\
Y \ar[r] & Y'
}
\]
in $\M$, where $f$ is both a monomorphism and a $\clD$-local equivalence. 
Let us show that $f'$ is a $\clD$-local equivalence. Let $Z$ be an object of
$\M$ and let $\tilde{Z}$ be a fibrant simplicial replacement of $Z$. By
applying the functor $\Hom_C(-, \tilde{Z})$, we obtain
a cartesian square
\[
\xymatrix@C=2pc@R=2pc{
\Hom_\M(Y', \tilde{Z}) \ar[d]_{\Hom_\M(f', \tilde{Z})\,} \ar[r] & \Hom_\M(Y, \tilde{Z})
\ar[d]^{\,\Hom_\M(f, \tilde{Z})} \\
\Hom_\M(X', \tilde{Z}) \ar[r] & \Hom_\M(X, \tilde{Z})
}
\]
of simplicial sets. Since every object of $\M$ is cofibrant, for every
object $W$ of $\M$, we have
\[ \RHom_\M(W, Z) \cong p(\Hom_\M(W, \tilde{Z})), \]
where $p : \pref{\Delta} \to \Ho(\pref{\Delta})$ is the localization
functor. Moreover, by Theorem~16.5.2 of~\cite{HirschMC}, the objects of this
square are fibrant and the morphism $\Hom_\M(f, \tilde{Z})$ is a fibration.
This shows that this square is homotopically cartesian. In particular, if
the morphism~$\Hom_\M(f, \tilde{Z})$ is a weak equivalence, then so is the
morphism $\Hom_\M(f', \tilde{Z})$. In other words, if $f$ is a $Z$-local
equivalence, then so is~$f'$. This shows that $f'$ is a $\clD$-local
equivalence.

The stability under transfinite composition follows from a similar argument,
the key point being that a limit defining a transfinite cocomposition in
which every object is fibrant and every morphism is a fibration is actually
a homotopy limit.
\end{proof}

\begin{prop}\label{prop:loc_fib_obj}
Let $A$ be a small category, let $\W$ be an accessible $A$-localizer and
let~$\Wp$ be an accessible $A$-localizer generated by $\W$ and a class of
morphisms $\clC$. Then the $\Wp$\nbd-fibrant objects are the $\W$-fibrant
$\clC$-local objects.
\end{prop}

\begin{proof}
We will denote by $\M$ (resp.~by $\M'$) the $\W$-model category (resp.~the
$\Wp$-model category).

Let $Z$ be a $\Wp$-fibrant object. Since $\W$ is included in $\Wp$, $Z$ is
also $\W$-fibrant. Let us show that $Z$ is $\clC$-local.
Let $X \to Y$ be an element of $\clC$. We have to show that the morphism
\[ \RHom_{\M}(Y, Z) \to \RHom_{\M}(X, Z) \]
is an isomorphism. By Lemma \ref{lemma:cmp_RHom}, this is the case if and
only if the morphism
\[ \RHom_{\M'}(Y, Z) \to \RHom_{\M'}(X, Z) \]
is an isomorphism, which is true since $X \to Y$ is a $\W'$\nbd-equivalence.

Suppose now $Z$ is a $\clC$-local $\W$-fibrant object and let us prove that
$Z$ is $\Wp$-fibrant. Let $f : U \to V$ be a trivial cofibration of $\M'$. We
have to show that
\[ \Hom_{\pref{A}}(V, Z) \to \Hom_{\pref{A}}(U, Z) \]
is a surjection. By Proposition 16.6.4 of \cite{HirschMC}, there exist
in $\Homi(\Delta, \M)$ cofibrant cosimplicial replacements $\tilde{U}$ and
$\tilde{V}$ of $U$ and $V$, respectively, and a cofibration
$\tilde{f} : \tilde{U} \to \tilde{V}$ such that $\tilde{f}_0 = f$. It thus
suffices to show that the morphism of simplicial sets
\[ \Hom_{\pref{A}}(\tilde{V}, Z) \to \Hom_{\pref{A}}(\tilde{U}, Z) \]
is a trivial fibration. Since $\tilde{f}$ is a cofibration between
cosimplicial replacements of objects of $\M$ and $Z$ is $\W$-fibrant,
Theorem 16.5.2 of \cite{HirschMC} implies that this morphism is a fibration.

Let us show that it is a weak equivalence. By hypothesis, the class of
$Z$-local equivalences contains $\W$ and $\clC$. Moreover, by Lemma
\ref{lemma:Loc_loc}, this class is an $A$-localizer. It follows by
definition of $\Wp$ that every $\Wp$-equivalence is a $Z$-local equivalence.
In particular, this shows that $f$ is a $Z$-local equivalence and so, by 
paragraph \ref{paragr:comp_RHom}, that the morphism
\[ \Hom_{\pref{A}}(\tilde{V}, Z) \to \Hom_{\pref{A}}(\tilde{U}, Z) \]
is a weak equivalence, thereby proving the result.
\end{proof}

\begin{paragr}
Let $\M$ be a model category and let $\clC$ be a class of morphisms of $\M$.
The \ndef{left Bousfield localization of $\M$ with respect to $\clC$} (if it
exists) is the unique model category structure on the underlying category of
$\M$ whose weak equivalences are the $\clC$-local equivalences and whose
cofibrations are the cofibrations of $\M$.
\end{paragr}

\begin{prop}\label{prop:loc_Bous_loc}
Let $A$ be a small category, let $\W$ be an accessible $A$-localizer and
let~$\Wp$ be an accessible $A$-localizer generated by $\W$ and a class of
morphisms $\clC$. Then the $\Wp$\nbd-model category is the left Bousfield
localization of the $\W$-model category with respect to $\clC$.
\end{prop}

\begin{proof}
We have to show that the $\Wp$-equivalences are exactly the $\clC$-local
equivalences. We will denote by $\M$ (resp.~by $\M'$) the $\W$-model
category (resp.~the $\Wp$-model category).

By Lemma \ref{lemma:Loc_loc}, the class of $\clC$-local equivalences is an
$A$-localizer. Since it contains $\W$ and $\clC$, the definition of $\Wp$
implies that every $\Wp$-equivalence is a $\clC$-local equivalence.

Conversely, let $f : X \to Y$ be a $\clC$-local equivalence. Let us show
that $f$ is a $\Wp$-weak equivalence. By Proposition \ref{prop:we_M_loc}, it
suffices to prove that for every $\Wp$-fibrant object~$Z$, the morphism
\[ \RHom_{\M'}(Y, Z) \to \RHom_{\M'}(X, Z) \]
is an isomorphism. By Lemma \ref{lemma:cmp_RHom}, these $\RHom$ can be
computed in $\M$. Moreover, by Proposition \ref{prop:loc_fib_obj}, the
object $Z$ is $\clC$-local. The above morphism is thus an isomorphism,
thereby proving the result.
\end{proof}

\begin{rem}
The proof of the above proposition only relies on the easy implication of
Proposition \ref{prop:loc_fib_obj}. Therefore, we could have proved it
directly. Proposition \ref{prop:loc_fib_obj} would then have been a
corollary of the theory of Bousfield localization (see Proposition 3.4.1 of
\cite{HirschMC}).
\end{rem}

\bibliographystyle{amsplain}
\bibliography{biblio}

\end{document}